\documentclass[11pt]{amsart}   
\usepackage{latexsym}
\usepackage{amsfonts}
\usepackage{amsmath,amsthm,amssymb}
\usepackage{fullpage}   

\newcommand{\cal}{\mathcal}

\def\epsilon{\varepsilon}
\def\phi{\varphi}

\def\hat{\widehat}

\newcommand{\Out}{\mbox{\rm Out}}
\newcommand{\Aut}{\mbox{\rm Aut}}

\newcommand{\FN}{F_N}   

\newcommand{\R}{\mathbb R}

\newcommand{\N}{\mathbb N}

\def\strutdepth{\dp\strutbox}
\def \ss{\strut\vadjust{\kern-\strutdepth \sss}}
\def \sss{\vtop to \strutdepth{
\baselineskip\strutdepth\vss\llap{$\diamondsuit\;\;$}\null}}

\def\strutdepth{\dp\strutbox}
\def \sst{\strut\vadjust{\kern-\strutdepth \ssss}}
\def \ssss{\vtop to \strutdepth{
\baselineskip\strutdepth\vss\llap{$\spadesuit\;\;$}\null}}

\def\strutdepth{\dp\strutbox}
\def \ssh{\strut\vadjust{\kern-\strutdepth \sssh}}
\def \sssh{\vtop to \strutdepth{
\baselineskip\strutdepth\vss\llap{$\heartsuit\;\;$}\null}}

\def\qed{\hfill\rlap{$\sqcup$}$\sqcap$\par}
\def\bar{\overline}
\def\tilde{\widetilde}

\def\strutdepth{\dp\strutbox}
\def \ss{\strut\vadjust{\kern-\strutdepth \sss}}
\def \sss{\vtop to \strutdepth{
\baselineskip\strutdepth\vss\llap{$\diamondsuit\;\;$}\null}}

\def\strutdepth{\dp\strutbox}
\def \sst{\strut\vadjust{\kern-\strutdepth \ssss}}
\def \ssss{\vtop to \strutdepth{
\baselineskip\strutdepth\vss\llap{$\spadesuit\;\;$}\null}}

\def\qed{\hfill\rlap{$\sqcup$}$\sqcap$\par}

\def\¶{\partial}

\newtheorem{thm}{Theorem}[section]
\newtheorem{cor}[thm]{Corollary}
\newtheorem{lem}[thm]{Lemma}
\newtheorem{prop}[thm]{Proposition}
\theoremstyle{definition}
\newtheorem{defn}[thm]{Definition}
\newtheorem{example}[thm]{Example}

\newtheorem{rem}[thm]{Remark}

\newtheorem{defn-rem}[thm]{Definition-Remark}
\newtheorem{hypothesis}[thm]{Hypothesis}
\newtheorem{convention}[thm]{Convention}
\theoremstyle{remark}

\numberwithin{equation}{section}

\begin{document} 
 
\author[M.~Lustig]{Martin Lustig}
\address{\tt 
Aix Marseille Universit\'e, CNRS, Centrale Marseille, I2M UMR 7373,
13453  Marseille, France
}
\email{\tt Martin.Lustig@univ-amu.fr}

 
\title[Partial train tracks]{Train track maps for graphs of groups}
 
\begin{abstract} 
We define train track maps for graphs-of-groups $\cal G$ and exhibit the precise conditions under which the fundamental finiteness properties known 
for classical train track maps extend to this generalization. These finiteness properties are the crucial tool to control the decrease of illegal turns under iteration of the train track map, and they are a principal ingredient in the answer to basic algorithmic questions about automorphisms induced by such train track maps on $\pi_1 \cal G$.
\end{abstract}

\subjclass[2010]{Primary 	20F65, 20F34, Secondary 20E36, 20E06, 20E08, 57M10}
 
\keywords{train track, free group automorphism, illegal turn}
 
\maketitle

\section{Introduction}
\label{intro}

Train track theory has first been introduced by Thurston for surface homeomorphisms, and later carried over by Bestvina-Handel in \cite{BH} to free group automorphisms. It has turned out to be a central tool in the study of outer automorphisms of free groups $\FN$ with finite rank $N \geq 2$. 

\medskip

{\em A self-map $f: \Gamma \to \Gamma$ of a finite graph $\Gamma$ is a train track map if every edge $e$ of $\Gamma$ is ``legal'': for any $t\geq 1$ the edge path $f^t(e)$ is reduced.}

\medskip

A special role in train track theory play 
the so called 
``INP''s\footnote{\,\!\! The acronym INP was originally an abbreviation for ``irreducible Nielsen path'', but by now is mainly used as icon for certain more general phenomena.}:
We say that a path $\eta$ in $\Gamma$ is an {\em INP path} if it can be written as concatenation $\eta = \gamma \circ \gamma'$ of two legal paths $\gamma$ and $\gamma'$, and for some integer $t \geq 1$ the path $f^t(\eta)$ is homotopic relative endpoints to $\eta$. If the map $f$ is {\em expanding}, in that no edge is contracted or mapped periodically by $f$, 
then the above path $\gamma \circ \gamma'$ can not be a {\em legal concatenation}, as at the concatenation point of any image path $f^t(\eta) = f^t(\gamma) \circ f^t(\gamma')$ there will be a non-trivial subpath which cancels out when $f^t(\eta)$ is reduced. 
The illegal concatenation point of $\gamma \circ \gamma'$ is called the {\em tip} of the INP path $\eta$.

\smallskip

One of the fundamental properties which makes expanding train track maps such a valuable tool is the fact that for every path or loop $\gamma$ in $\Gamma$ there is an exponent 
$t(\gamma) \geq 0$ such that after reduction the image path $f^{t(\gamma)}(\gamma)$ becomes {\em pseudo-legal}\,: it is a legal concatenation of legal paths and 
INP paths. 
Furthermore the following holds, which is the first of three fundamental finiteness properties of expanding train track maps:

\medskip

{\em There are only finitely many INP paths in $\Gamma$.}

\medskip

The second fundamental finiteness property concerns the number of times, given by the above exponent $t(\gamma)$, for the iterations of $f$ needed in order to transform $\gamma$ into a path which is homotopic to a pseudo-legal path:

\medskip

{\em 
The exponent $t(\gamma)$ does not depend on the specific path $\gamma$ itself, but only on the number of ``illegal turns'' in $\gamma$: these are the concatenation points in the canonical decomposition of $\gamma = \gamma_1 \circ \gamma_2 \circ \ldots \circ \gamma_q$ as concatenation of maximal legal subpaths $\gamma_i$.
}

\medskip

The third fundamental finiteness property concerns the speed under which illegal turns in any path disappear, under iteration of $f$ and subsequent reduction. Of course, once the reduction $[f^t(\gamma)]$ has become pseudo-legal, the number of illegal turns will stay constant, since any illegal turn at the {\em tip} of an INP path will stay illegal forever. However, 
through counting properly (see Proposition \ref{from-LUbig} (3) for the precise statement), we obtain: 

\medskip

{\em 
For any path or loop $\gamma$ the number of illegal turns in $\gamma$ decreases exponentially under iteration of $f$, if one only counts those illegal turns which are not (or will not eventually become, under iteration of $f$ and subsequent reduction) the tip of an INP subpath of the reduced path $[f^{t(\gamma)}(\gamma)]$.
}

\medskip

The purpose of this paper is to show that the analogous facts are true for any expanding train track map $f: \cal G \to \cal G$ of a graph-of-groups $\cal G$ with trivial edge groups, 
and with vertex groups that are not necessarily free, a long as 
some basic conditions are assumed. 
The 
main task of our efforts 
in the subsequent sections 
is the 
detailed exhibition of these ``basic conditions''
(see Hypothesis \ref{running-hyp}), 
and to show how 
precisely 
they interfere in the proof of the finiteness properties. 
Our results allow applications 
which go quite a bit beyond free groups; for example, it is to be expected that they will be useful for 
the study of 
automorphisms of free products, 
which have recently seen anew 
a lot of activity and interest (see 
\cite{DahL-19, FrMa-13, GuiHo-17, GuiHo-19, HeKi-16, Ho-14a, Ho-14b, Sy-16}).

\smallskip

All of the arguments used in this paper are  
elementary, but at times a bit intricate; the main idea for the proof goes back to the proof of Lemma 3.2 
in 
\cite{Lu_conj-pr1_MPI}. The precise definition of the terms used in the following theorems are given in 
sections \ref{prelims} and \ref{sec-turns}; for the convenience of the reader we state here 
slightly streamlined versions 
of our results which do  
not quite reflect the full strength of 
what is 
proved in sections
\ref{sec:INP-paths}, \ref{sec:INP-candidates} and \ref{sec:algorithm}.

\begin{thm}
\label{main}
Let $\cal G$ be a 
graph-of-groups with trivial edge groups
and 
let $f: \cal G \to \cal G$ be an expanding train track map. 
Assume furthermore that 
$\pi_1\cal G$ is finitely generated, 
and that $f$ 
induces an automorphism on $\pi_1 \cal G$.  Then the following holds:
\begin{enumerate}
\item
There are only finitely many INP paths in $\cal G$.
\item
For every edge path or loop $\gamma$ in $\cal G$ there is an exponent 
$t(\gamma) \geq 0$ such that the reduced path $[f^{t(\gamma)}(\gamma)]$ derived from $f^{t(\gamma)}(\gamma)$ is a legal concatenation of legal and INP subpaths. 
The exponent $t(\gamma)$ depends only on the number of illegal turns in $\gamma$ and not on the particular choice of $\gamma$ itself. 
\item
There is an exponent $\hat t \geq 1$ such that for any 
edge path or loop 
$\gamma$ with legal iteration $[f^{t(\gamma)}(\gamma)]$ 
the number of illegal turns in 
$[f^{\hat t}(\gamma)]$ 
is at most half the number of illegal turns in $\gamma$.
\end{enumerate}
\end{thm}

For a classical train track map $f: \Gamma \to \Gamma$ it is well known that the finite set of INP paths can be determined from the data that describe the map $f$. The analogue is true in the graph-of-groups setting, 
if there is enough algorithmic information available for the vertex groups. 
The precise conditions needed are explicited in Hypothesis \ref{given-data}; we state here only the case which we are most interested in:

\begin{thm}
\label{algo-det}
Let $\cal G$ and $f$ be as in Theorem \ref{main}, and assume that 
$\pi_1 \cal G \cong \FN$ is free. We also assume that  
for every vertex group $G_v$ of $\cal G$ the induced map on $\pi_1 G_v$ is given by a finite set of 
data.

Then there exist an algorithm which determines the following:
\begin{enumerate}
\item
The set of all INP paths in $\cal G$.
\item
For any edge path or loop $\gamma$ the exponent $t(\gamma)$ 
from part (2) of Theorem \ref{main}.
\item
The iteration exponent $\hat t$ from part (3) of Theorem \ref{main}.
\end{enumerate}
\end{thm}

The proof of this theorem is given 
in section \ref{sec:algorithm}, where we use our careful analysis of the above named ``basic conditions'' to derive the necessary finiteness ingredients and their algorithmic 
determination: 
in section \ref{sec:algorithm} 
we describe a combinatorial procedure which consists of 5 steps that are explicitly carried through.

These results can be used for many algorithmic purposes, and in particular for the construction of a particularly well suited train track representative for any automorphism $\phi \in \Out(\FN)$. A short discussion is given in section \ref{sec:algorithm}, were we also give a sample application in Corollary \ref{fixed-conj}
(concerning fixed conjugacy classes).
This enables us in section \ref{sec:tts-with-periodics} to digress into a 
brief study 
of absolute train track maps - with and without periodic edges - and to deduce some 
unexpected facts about them 
(see 
Remark \ref{periodic-non-fixed} (2) and Example \ref{exotic-near-iwip}).

\medskip
\noindent
{\em Acknowledgements:}
The author would like to thank Ilya Kapovich for encouraging remarks which helped to bring the paper into a much more satisfying final shape. 
Further thanks go to Jean Pierre Mutanguha for an inspiring email conversation regarding some of the content in section \ref{sec:tts-with-periodics}. 

\section{Graphs-of-spaces, paths and edge paths, turns}
\label{prelims}

One of the basic problems when dealing with INP paths, known already from the classical setting where all vertex groups are trivial, 
comes from 
the fact that in general 
an endpoint of an INP path $\eta$ is not a vertex, but a point that lies in the interior of an edge. Thus a very careful treatment of the notion of paths and edge paths for the graphs in question is necessary; this is the prime reason why we work for 
most of this 
paper with graphs-of-spaces rather than graphs-of-groups. The transition between these two settings is canonical; nevertheless we will start this section by setting it up with all due care.

\smallskip

Let us denote by $\cal G_X$ be a graph-of-groups with trivial edge groups, built on a finite connected graph, 
with a vertex group $G_v$ for any vertex $v$ of the graph.
Let $\cal G$ be a {\em topological realization} of the graph-of-groups $\cal G_X$ as a {\em graph-of-spaces}, 
by which we mean that every vertex group $G_v$ of $\cal G_X$ is realized by a {\em vertex space} $X_v$ with $\pi_1 X_v = G_v$. We can thus think of $\cal G_X$ as obtained from 
$\cal G$ by 
contracting every connected component $X_v$ of the {\em relative part} $X = \cup X_v$ of $\cal G$ to a single vertex $v$, and in turn providing $v$ with 
the 
vertex group $G_v$. 
Similarly, the Bass-Serre tree $\tilde{\cal G}_X$ associated to $\cal G_X$ is obtained from the universal covering $\tilde{\cal G}$ of $\cal G$ 
by contracting every connected component $X_{\tilde v}$ of the full lift $\tilde X \subset \tilde{\cal G}$ of $X$ to a single vertex $\tilde v$.
This gives
$$\pi_1 \cal G = \pi_1 \cal G_X \, ,$$
where the term on the left hand side is an ordinary fundamental group, while on the right hand side we have the classical fundamental group of a graph-of-groups. The language for graph-of-groups used here is standard; it is 
conveniently 
summarized for instance in \S2.1 of \cite{LY}.

\smallskip

Even though 
in many occasions 
a vertex space $X_v$ may well 
be a cell complex and thus 
also contain edges, we will consider 
such cells 
only as 
``local cells'' (or
in particular 
as 
``local edges''); 
by an {\em edge of $\cal G$} 
(sometimes called a ``graph-of-spaces edge'') 
we always mean an edge in $\cal G \smallsetminus X$. Indeed, the concrete shape of the vertex spaces never plays a role, and 
in particular 
we will consider 
any path in $\cal G$ which traverses a vertex space $X_v$ only up to homotopy 
within $X_v$.

\medskip

An {\em edge path} in $\cal G$ is a path $ \gamma = e_1 \circ \chi_1 \circ e_2 \circ \chi_2 \circ \ldots \circ \chi_{r-1} \circ e_r$, where each $e_i$ is an edge in $\cal G \smallsetminus X$, while every {\em connecting path} $\chi_i$ is contained in 
a 
vertex space $X_{v(i)}$ which 
also 
contains the terminal endpoint of the edge $e_i$ and the initial endpoint of the edge $e_{i+1}$.
The 
above described 
canonical transition between $\cal G$ and $\cal G_X$ 
transforms $\gamma$ into a 
{\em connected} 
word $W_\gamma = e_1 \, \frak x_1 \, e_2 \, \frak x_2 \,  \ldots \, \frak x_{r-1} \,  e_r$ in the Bass group $\Pi(\cal G_X)$, i.e.  each $\frak x_i$ is an element of the vertex group $G_{v(i)}$, for $v(i) 
= \tau(e_i) = \tau(\bar e_{i+1})$. Here $\bar e$ denotes the edge $e$ with reversed orientation, and 
$\tau(e)$ is the terminal vertex of $e$ 
in the graph 
that underlies $\cal G_X$. The {\em length} $|\gamma|$ of an edge path $\gamma$ is equal to the number of edges (from $\cal G \smallsetminus X$) traversed by $\gamma$,
so that one has $|\gamma| = |W_\gamma|$ for the usual combinatorial length $|W|$ of any word $W \in \Pi(\cal G_X)$.

To be unequivocal, we'd like to state explicitly the following convention. 
As a consequence of it, two edge paths $\gamma$ and $\gamma'$ in $\cal G$ are {\em equal} if and only if any lifts 
$\tilde W_\gamma$ and $\tilde W_{\gamma'}$ 
of them to the Bass-Serre tree $\tilde{\cal G}_X$ agree, once they have the same first edge. 

\begin{convention}
\label{edge-path}
In this paper we use the convention that 
any 
edge path 
$\gamma$ in $\cal G$ starts and finishes with an edge from $\cal G \smallsetminus X$. In particular, its length always satisfies $|\gamma| \geq 1$.
\end{convention}

A {path} $\gamma'$ in $\cal G$ is 
said to be 
{\em zero} if it is contained in any vertex space $X_v$ of $\cal G$. A {\em non-zero path} 
$\gamma'$ 
is 
not necessarily an edge path, but we postulate that such $\gamma'$ is always part of an edge path $\gamma$ as specified above, where some (possibly trivial) initial segment from the first and terminal segment from the last edge of $\gamma$ is missing in $\gamma'$. 
These missing edge segments are never equal to the whole edge, so that 
any non-zero path $\gamma'$ always 
starts and ends in a 
non-trivial 
edge segment 
(which could be the whole edge), 
and not in a connecting path $\chi_i$.
The edge path $\gamma$ is called 
the {\em canonical vertex-prolongation} of the path $\gamma'$. 
It is well defined by 
the path 
$\gamma'$, 
since $\gamma'$ is  assumed to be non-zero.

\begin{convention}
\label{path}
By a {\em path} $\gamma'$ in $\cal G$ we always mean a zero path or a non-zero path as defined above. 
The restrictions and their consequences which are implied by this convention are discussed in Remark \ref{pseudo-paths} below.
\end{convention}

For a zero path $\gamma'$ 
we define 
the {\em length} $|\gamma'|$ of $\gamma'$ to be equal to $0$; if $\gamma'$ is non-zero, we use the canonical vertex-prolongation 
$\gamma$ of $\gamma'$ 
to set $|\gamma'| = |\gamma|$.

With these conventions, a path could also be a single point in the interior of an edge; in this case its length would be 1 (and not 0). 
Of course, every edge path in $\cal G$ is in particular a path, and any path $\gamma$ is an edge path if and only if it coincides with its canonical vertex-prolongation.

\begin{rem}
\label{pseudo-paths}
(1)
The specification in Convention \ref{path} of our usage of the word ``path'' only for zero or non-zero paths has the consequence, that a concatenation $\gamma \circ \chi$ of a non-zero path $\gamma$ with a zero path $\chi$ is not a ``path'' in the terminology set up for this paper. Indeed, such {\em path concatenations} occur only very sparsely in this paper, and any such occurrence will be pointed out specifically by our terminology. 

\smallskip
\noindent
(2)
On the other hand, our insistence on the particular conditions 
spelled out above 
for what we except as a ``path'', will turn out to be vital for the precise arguments put forward in this paper (compare Remark \ref{purity}).

\smallskip
\noindent
(3)
Loops which are not contained in one of the vertex spaces do occur 
occasionally (but not often) 
in this paper. We would like to point out that, according to our conventions, such a loop 
$\gamma$
can in general not be represented by an edge path, but rather by a path concatenation $\gamma_0$ as in (1) above. Specifically, in this case one has a concatenation
\begin{equation}
\label{cyclic-concat}
\gamma_0 = e_1 \circ \chi_1 \circ e_ 2 \circ \chi_2 \ldots \chi_{q-1} \circ e_q \circ \chi_q
\end{equation}
of an edge path 
$e_1 \circ \chi_1 \circ e_ 2 \circ \chi_2 \ldots \chi_{q-1} \circ e_q$ with some zero path $\chi_q$ 
that connects the terminal endpoint of $e_q$ to the initial endpoint of $e_1$. 
As usual, two such 
closed path 
concatenations $\gamma_0$ and $\gamma'_0$ are considered to represent the same loop $\gamma$ in $\cal G$ if 
and only if $\gamma_0$ and $\gamma'_0$ are identical 
up to a cyclic permutation of the indices in (\ref{cyclic-concat})
(and, as usual, modulo homotopy in the vertex spaces).
Thus, if there is no danger of confusion, we will 
in the sequel 
not always distinguish notationally between 
the loop $\gamma$ and and the closed path concatenation $\gamma_0$ 
as above.
\end{rem}

We  define a {\em turn} in $\cal G_X$ to be a word $e \,\frak x \,  e' \in \Pi(\cal G_X)$ with $\frak x \in G_{\tau(e)} = G_{\tau(\bar e')}$. 
Equivalently, a turn 
in $\cal G$ 
is 
an edge path $\gamma = e \circ \chi\circ e'$ of length 2 in $\cal G$. This definition is a slight variation of classical train track terminology, where one has always $\frak x = 1$ and hence prefers to present the path $e \, e'$ as pair $(\bar e, e')$.
We say that an edge path $\gamma_0$ {\em uses} a turn $e \circ \chi\circ e'$ if $\gamma_0$ contains the turn as subpath.
Since we will need to consider frequently paths that are not edge paths, it is useful to extend the notion of a ``turn'' 
and of ``using a turn'' 
also to paths $e \circ \chi\circ e'$ where $e$ and $e'$ are non-trivial edge segments.

A turn $e \,\frak x \,  e'$ is {\em degenerate} if $\frak x = 1 \in G_{\tau(e)}$ and $\bar e = e'$. Equivalently, for the corresponding edge path $e \circ \chi \circ e'$  one has $\bar e = e'$ and $\chi$ is a contractible loop in $X_{\tau(e)}$.
An edge path $\gamma$ in $\cal G$ is {\em reduced} if any subpath of length 2 defines a non-degenerate turn. A path $\gamma'$ in $\cal G$ is {\em reduced} if the canonical vertex-prolongation $\gamma$ of $\gamma'$ is reduced, 
or if 
$\gamma'$ is zero.
This includes the special case where $\gamma'$ has both endpoints 
$x$ and $y$ on the same edge $e$ of $\cal G$ and $\gamma'$ is equal to the segment 
$[x, y]$ 
of $e$. 

Every non-reduced 
path $\gamma$ can be transformed by iterative reductions (i.e. cancellation of degenerate turns) into a reduced path $[\gamma]$ with same endpoints as $\gamma$. 
If such a reduction takes place at the beginning or the end of an edge path, the ``left-over'' connecting path at the beginning or end, after the reduction, is also cancelled, so that the 
path $[\gamma]$ 
which results from reducing any edge path $\gamma$ 
is either zero, or else it starts and ends 
in an edge from $\cal G \smallsetminus X$ 
and is thus again an edge path.

The order of the reductions is irrelevant: the reduced path $[\gamma]$ is uniquely determined by $\gamma$, and the two paths are homotopic in $\cal G$ relative endpoints 
(or rather, relative to homotopy of the endpoints in $X$). 
If $\gamma$ is a path but not an edge path, then we have to 
admit in the reduction process also the cancellation of 
degenerate turns $e \circ \chi \circ e'$ where $\bar e = e'$ is an edge segment.
A path $\gamma$ which can be reduced to a zero path 
or to a single point in the interior of an edge 
is called a {\em backtracking path}.
We note that if $\gamma_0$ is any maximal backtracking subpath of 
any path $\gamma$, 
then $\gamma_0$ 
is 
necessarily 
an edge path, 
except possibly if $\gamma_0$ is an initial or terminal subpath of $\gamma$.

\begin{convention}
\label{ignore}
Since the transition between $\cal G$ and $\cal G_X$ is completely canonical, 
we will from now on allow ourselves to pass from one to the other without always explicitly notifying the reader. We will freely use both languages, according to whichever is better suited to the circumstances. In order to help avoiding a potential confusion, we will make the effort to use the term ``vertex'' only in the $G_X$ environment, while for example paths in $\cal G$ will terminate in ``endpoints'' (which may or may not be contained in 
a vertex space).
\end{convention}

\section{Graph-of-spaces morphisms, cancellation bound, image turns, special paths}
\label{sec-turns}

Any map $f: \cal G \to \cal G$ in this paper is always assumed to be a {\em graph-of-spaces morphism}, i.e. it maps
vertex spaces to vertex spaces, and edges to edge paths. 
The map $f$ induces on the associated graph-of-groups $\cal G_X$ a map $f_X: \cal G_X \to \cal G_X$.
In accordance with Convention \ref{ignore} we 
sometimes denote the map $f_X$ 
also by $f$, so that for example $f(v)$ denotes the image of 
any vertex $v$ 
in 
the graph that underlies $\cal G_X$. 
The map $f$ is {\em expanding} if some power $f^t$ maps every edge $e$ to an edge path of length $|f^t(e)| \geq 2$. 
Notice however that, even if $f$ is not expanding, then according to our Convention \ref{edge-path} the edge path $f(e)$ is 
a non-zero path and hence always has length $|f(e)| \geq 1$.

We 
are most interested in 
the case where 
$f$ induces an (outer) automorphism of $\pi_1 \cal G$, but in order to include also more general situations, we will now list the two weaker assumptions which are sufficient to deduce the results of this paper:

\begin{hypothesis}
\label{running-hyp}
Let $f: \cal G \to \cal G$ be 
a graph-of-spaces morphism. 
We consider the following two 
conditions:
\begin{enumerate}
\item
The 
map $f_*$ induced by $f$ on $\pi_1 \cal G$ is injective. In particular, 
for any vertex space $X_v$ of $\cal G$ the map on $\pi_1 X_v$ induced by $f$ is injective.
Furthermore, 
$f$ permutes the {\em essential} vertex spaces of $\cal G$, 
by which we mean the vertex spaces $X_v$ with non-trivial $\pi_1 X_v$.
\item
The map $f$ possesses 
a {\em cancellation bound} 
$C \geq 0$: 
the length of any backtracking subpath 
in the $f$-image of any reduced path 
is bounded above by $C$. 
\end{enumerate}
\end{hypothesis}

\begin{lem}
\label{C-bound}
Let $f: \cal G \to \cal G$ be 
graph-of-spaces morphism 
which induces an automorphism 
$f_*$ on 
$\pi_1 \cal G$.
Let us also assume 
that $\pi_1 \cal G$ is 
finitely generated. 
Then both conditions (1) and (2) of Hypothesis \ref{running-hyp} are satisfied.
\end{lem}

\begin{proof}
(1)
Since $\pi_1 \cal G$ is assumed to be finitely generated, 
then by Grushko's theorem the same is true 
for each of the vertex group $\cal G_v$, and there is an (up to conjugation) canonical f.g. subgroup $\cal G_v^*$ which 
doesn't split into a free product with a non-trivial free group as factor. 
Similarly, there is 
a complementary 
free group $F_{N(v)}$ 
of finite rank $N(v)$ 
such that $\cal G_v = G_v^* * F_{N(v)}$. Here of course $\cal G_v^*$ or $F_{N(v)}$ may well be trivial.

Since $f$ is a graph-of-spaces morphism, 
the induced map $f_*$ 
maps each of the vertex groups $\cal G_v$ 
to a conjugate of 
some $\cal G_{f(v)}$. It follows from 
Kurosh's subgroup theorem 
that $f_*$ 
maps each of the groups $\cal G_v^*$ to a conjugate of 
$\cal G^*_{f(v)}$. 
Since $f_*$ is assumed to be an automorphism, this map $\cal G_v^* \to \cal G^*_{f(v)}$ must be an automorphism, and $f$ must permute those vertices $v$ with non-trivial $\cal G_v^*$.

Through quotienting out the $\cal G_v^*$ 
we can deduce now 
the same conclusion 
for the factors $F_{N(v)}$. It follows that $f$ must permute the essential vertex spaces of $\cal G$.

\smallskip
\noindent
(2)
The given train track map $f$ thus permutes the essential vertices 
of the associated graph-of-groups $G_X$. We now perturb $f$ by an isotopy which only involves the non-essential vertices and their adjacent edges, to obtain a map $f': \cal G \to \cal G$ which 
acts as permutation on {\em all} of the vertices of $\cal G_X$. 
The resulting map $f'$ 
is still a graph-of-space morphism, and since it is isotopic to $f$, the map induced by $f'$ on $\pi_1 \cal G$ is unchanged and hence still an automorphism. The above isotopy can be achieved by moving every vertex along an edge path of length at most the diameter $D \geq 0$ of the finite graph that underlies $\cal G$.

We now consider the Bass-Serre tree $\tilde{\cal G}_X$ of the graph-of-groups $\cal G_X$ associated to $\cal G$, and note that 
both, $f$ and  $f'$, 
lift to maps $\tilde f_X : \tilde{\cal G}_X \to \tilde{\cal G}_X$ and $\tilde f'_X : \tilde{\cal G}_X \to \tilde{\cal G}_X$  respectively. 
From the assumption that $f_* = f'_*: \pi_1\cal G \to \pi_1 \cal G$ is an automorphism we know that the permutation by $f'$ on the vertices of $\cal G_X$ from the previous paragraph implies that $\tilde f'_X$ restricts 
on the set of vertices of $\tilde{\cal G}_X$ to a bijection. 

Hence, if we provide every edge of $\tilde{\cal G}_X$ with length 1, the map $\tilde f'$ becomes a quasi-isometry of $\tilde{\cal G}_X$. In particular, every geodesic $\gamma$ in $\tilde{\cal G}_X$ is mapped to a quasi-geodesic $\tilde f'_X(\gamma)$ in $\tilde{\cal G}_X$. Since $\tilde{\cal G}_X$ is a tree, and thus a 0-hyperbolic metric space, it follows that the quasi-geodesic $\tilde f'_X(\gamma)$ must travel within a bounded neighborhood of the unique geodesic $\gamma'$ which has the same endpoints as $\tilde f'_X(\gamma)$. Hence every vertex on $\tilde f'_X(\gamma)$ is of distance at most $C'$ from $\gamma'$, for some constant $C' \geq 0$ independent of the choice of $\gamma$. 

Considering now again the map $f$ and its lift $\tilde f_X$, we deduce that every vertex on $\tilde f_X(\gamma)$ is of distance at most 
$C' + D + 2D$ from $\gamma'$ (where the last term in the sum is due to the fact that $\tilde f_X(\gamma)$ and $\tilde f'_X(\gamma)$ may not have the same endpoints).
But then $C = C' + 3D$ is precisely a cancellation bound as required in Hypothesis \ref{running-hyp} (2).
\end{proof}

\medskip

Any 
graph-of-spaces morphism $f: \cal G \to \cal G$ defines an induced map on the set of turns in $\cal G$ as follows 
(where we keep in mind that $|f(e)| \geq 1$ for any edge $e$): 

Let $e_1 \circ \chi \circ e_2$ be any turn on $\cal G$. Then the turn $e'_1 \circ \chi' \circ e'_2$ is the {\em image turn} of the turn $e_1 \circ \chi \circ e_2$ if $e'_1$ is the last edge of $f(e_1)$, if $e'_2$ is the first edge of $f(e_2)$, and if $f(\chi)$ is homotopic relative endpoints to $\chi'$. Correspondingly, the turn $e_1 \circ \chi \circ e_2$ is a {\em preimage turn} of the turn $e'_1 \circ \chi' \circ e'_2$, under the map $f$.

We verify directly from this definition:

\begin{rem}
\label{forgotten}
Let $\gamma$ be any edge path on $\cal G$, and let $\gamma' = [f(\gamma)]$ be the reduced image path. 

\smallskip
\noindent
(1)
If $\gamma$ is legal (see Definition-Remark \ref{legal+} (1) just below), or more generally, if $f(\gamma)$ is reduced 
(so that one has $f(\gamma) = \gamma'$), 
then every turn used by $\gamma$ is the preimage turn of some turn used by $\gamma'$.

\smallskip
\noindent
(2)
If $f(\gamma)$ is not reduced, the conclusion stated in (1) above is in general wrong.
\end{rem}

\begin{rem}
\label{D-one}
(1)
The above defined induced map on turns can be viewed alternatively from the following point of view:

One first defines a self-map $Df$ on the (finite) set of edges of $\cal G$ by declaring $Df(e)$ to be the first edge of the edge path $f(e)$. One then obtains the induced maps on turns (sometimes denoted by $D^2f$) via:
$$\bar e \circ \chi \circ e' \quad \mapsto \quad  \bar{Df(e)} \circ f(\chi) \circ D(e')$$

\smallskip
\noindent
(2)
From the finiteness of the edge set of $\cal G$ we see that the map $Df$ 
is 
on any edge $e$ eventually periodic. If all vertex groups of $\cal G$ are trivial, the same is true for the induced map $D^2f$ on turns, so that every legal turn is eventually periodic, and both, the maximal seize of a periodic orbit as well as the maximal number of iterations needed before a legal turn becomes periodic or an illegal turn becomes degenerate, is determined by the number of edges in $\cal G$.

\smallskip
\noindent
(3)
If $\cal G$ has non-trivial vertex groups, 
then in general there will be 
legal turns which have infinite orbits (i.e. they are not eventually periodic). Hence the question, whether there is an upper bound on the number of iterations needed before any illegal turn becomes degenerate, is more delicate. It will be addressed below in Lemma \ref{crucial-fin}.
\end{rem}

The following terminology is inherited from the case where all vertex groups are trivial:

\begin{defn-rem}
\label{legal+}
Let $f: \cal G \to \cal G$ be a graph-of-spaces morphism.
\begin{enumerate}
\item
A path $\gamma$ in $\cal G$ is {\em legal} if the image path $f^t(\gamma)$ is reduced, for 
any integer $t \geq 1$.
\item
The map $f$ is a {\em train track map} if every edge (understood as edge path of length 1) is legal.
\item
A turn is {\em legal} if the $f^t$-image turns 
are non-degenerate, for 
all 
$t \geq 1$. Otherwise the turn is called {\em illegal}.
\item
A path $\gamma$ 
turns out to be 
legal if and only if every turn used by $\gamma$ 
is legal.
\item
A concatenation $\gamma_0 = \gamma \circ \chi \circ \gamma'$ of two 
non-zero 
paths $\gamma$ and $\gamma'$ 
through a connecting zero path $\chi$ 
is 
called 
{\em legal} if the turn 
used 
by $\gamma_0$ 
at the concatenation vertex 
space 
is legal.
\item
A closed 
path 
concatenation $\gamma_0$ as in (\ref{cyclic-concat}) is called {\em cyclically legal} if every turn in any cyclic permutation of $\gamma_0$ is legal.
\end{enumerate}
\end{defn-rem}

The following statements are easy to derive from the above listed facts, but since they are crucial for the following sections, we carry through the proof.

\begin{lem}
\label{crucial-fin}
Let $f: \cal G \to \cal G$ be an expanding 
graph-of-spaces morphism, and assume that 
$f$ 
satisfies condition (1) of Hypothesis \ref{running-hyp}.
Then the following facts are true:
\begin{enumerate}
\item
For any two edges $e, e'$ of $\cal G$ there is at most one connecting path $\chi$ such that $e \circ \chi \circ e'$ has a degenerate image turn.
\item
Every non-degenerate $f$-periodic turn is legal.
\item
Every turn in $\cal G$ has only finitely many preimage turns.
\item
There are only finitely many illegal turns in $\cal G$.
\item
There exists an exponent $t_0 \geq 0$ such that for any illegal turn the $f^{t_0}$-image turn is degenerate.
\end{enumerate}
\end{lem}

\begin{proof}
We recall that $\cal G$ has only finitely may edges, and that connecting paths are considered to be equal if they only differ by a homotopy relative endpoints 
in their vertex spaces. Hence the claims (1) and (3) are a direct consequence of our injectivity hypothesis on the induced vertex group homomorphisms. Claim (2) follows directly from the definition of a ``legal'' turn. In order to prove claims (4) and (5), we invoke the map $Df$ 
from 
part (1) of Remark \ref{D-one} 
and 
observe for any turn $e \circ \chi \circ e'$, where $\bar e$ and $e'$ are assumed to be $Df$-periodic, that the turn is illegal if and only if one has $\bar e = e'$ and $\chi$ is contractible relative endpoints; in other words: if and only if the turn is degenerate. From Remark \ref{D-one} (2) we know that there is an exponent $t_0$ such that for any turn in $\cal G$ the $Df$-periodicity assumption in the previous sentence is true for the $f^{t_0}$-image turn, which proves assertion (5). Claim (4) then follows by applying assertion (3) $t_0$ times 
to any degenerate turn in $\cal G$.
\end{proof}

\medskip

The statements (4) and (5) of Lemma \ref{crucial-fin} are the first finiteness results in our setting that go beyond what is known for classical train track maps. A second new finiteness ingredient, which is crucially used in the next section, will be presented now. It is based on the notion of ``special'' and ``pre-special'' paths, which have been invented specifically for this purpose:

\begin{defn}
\label{special-words}
Let $f: \cal G \to \cal G$ be a graph of spaces morphism.
\begin{enumerate}
\item
For any integer $t \geq 1$ a turn in $\cal G$ 
is called {\em $t$-special}  if 
it 
is 
used by any of the  
edge path 
$f^{t}(e_0)$ or 
$f^{t}(e_1 \circ \chi \circ e_2)$, 
where $e_0, e_1$ and $e_2$ are edges of $\cal G$ and the vertex 
space 
$X_v$ 
which contains the connecting path $\chi$ 
is 
inessential (i.e. $\pi_1 X_v$ is trivial).
\item
Any 
legal 
path $\gamma$ (not necessarily an edge path) in $\cal G$ is called {\em pre-$t$-special} if 
every turn 
used by $f^t(\gamma)$ is $t$-special.
\end{enumerate}
\end{defn}

\begin{lem}
\label{zero-finite}
Let $f: \cal G \to \cal G$ be an expanding train track map, and assume that Hypothesis \ref{running-hyp} (1) is satisfied. Then one has:
\begin{enumerate}
\item
For any $t \geq 1$ there are only finitely many $t$-special turns in $\cal G$.
\item
For any integers $t \geq 1$ and $k \geq 0$ there are only finitely many 
legal 
edge paths $\gamma$ in $\cal G$ of length $|\gamma| \leq k$ which are pre-$t$-special.
\end{enumerate}
\end{lem}

\begin{proof}
Statement (1) follows directly from the properties listed in Definition \ref{special-words}, since there are only finitely many edges in $\cal G$ and since any inessential vertex space $X_v$ admits only finitely many distinct connecting paths (up to homotopy in $X_v\,$, as usual).

For statement (2) we invoke 
Remark \ref{forgotten} (1) 
and 
Lemma \ref{crucial-fin} (3), for $f^t$ in place of $f$, in order to deduce from (1) that there only finitely many turns that the path $\gamma$ can possibly use. It thus follows from the length restriction on $\gamma$ that there are only finitely many paths $\gamma$ which satisfy the conditions listed in (2).
\end{proof}

\begin{lem}
\label{special-applied}
Let $f: \cal G \to \cal G$ be 
as in Lemma \ref{zero-finite}, and 
let $\gamma$ and $\gamma'$ be two 
non-zero
legal paths with 
initial points that lie in a 
common vertex space.
Assume furthermore that $\gamma$ and $\gamma'$ have 
distinct first edges, 
and endpoints that may lie in the interior of an edge. 
If $\gamma$ or $\gamma'$ consist only of an edge segment, then we require that the edges which contain these segments are distinct, or distinct from the first edge of the other path.

Assume that for some integer $t \geq 1$ one has  $f^t(\gamma) = f^t(\gamma')$.
Then every 
turn in 
$f^t(\gamma) = f^t(\gamma')$ is $t$-special.
\end{lem}

\begin{proof}
We lift $\gamma$ and $\gamma'$ to paths $\tilde \gamma$ and $\tilde\gamma'$ in the Bass-Serre 
tree 
$\tilde{\cal G}_X$, which are chosen such that $\tilde \gamma$ and $\tilde\gamma'$ have a common initial vertex. 
By Hypothesis \ref{running-hyp} (1) 
any lift 
$\tilde f_X^t: \tilde{\cal G}_X \to \tilde{\cal G}_X$ of the map $f^t$ 
acts 
injectively 
on the 
set of 
essential vertices of $\tilde{\cal G}_X$ 
(by which we mean those vertices that are lifts of essential vertices in $\cal G_X$). 
Now 
every turn used by 
$\tilde f_X^t(\tilde \gamma) = \tilde f_X^t(\tilde \gamma')$, which is not used by the $\tilde f_X^t$-image of some edge from either $\tilde \gamma$ or $\tilde \gamma'$, must be the image turn of both, a turn in $\tilde \gamma$ and a turn in $\tilde \gamma'$. Since $\tilde \gamma$ and $\tilde\gamma'$ have a common initial vertex but distinct first edges (as assumed in the statement of the lemma), and since $\tilde{\cal G}_X$ is a tree, these two turns in $\tilde \gamma$ and $\tilde \gamma'$ must 
take place at distinct vertices of 
$\tilde{\cal G}_X$. But then it follows from the above 
injectivity 
on 
the 
essential vertices of $\tilde{\cal G}_X$ that 
one of those vertices 
must be inessential. It follows from the conditions listed in Definition \ref{special-words} (1) that in either case, every turn used by $f^t(\gamma) = f^t(\gamma')$ is $t$-special.
\end{proof}

\begin{lem}
\label{first-finiteness}
Let $f: \cal G \to \cal G$ be a expanding train track map, and assume that conditions (1) and (2) of Hypothesis \ref{running-hyp} are satisfied. 
For any $t \geq 1$ 
let $\cal V(f^t)$ be the set of edge paths $\eta = \bar \gamma_1 \circ \chi \circ \gamma_2$ 
in $\cal G$ 
with the following properties:

Assume that $\gamma_1$ and $\gamma_2$ are two 
legal edge paths with 
initial points in the vertex space $X_v$ which also contains the connecting path $\chi$, and that $\gamma_1$ and $\gamma_2$ have distinct first edges. Assume furthermore that $f^t(\chi)$ is contractible in 
$X_{f^t(v)}$, and define $\gamma'_1$ and $\gamma'_2$ to be the initial subpaths of $\gamma_1$ and $\gamma_2$ respectively which satisfy $f^t(\gamma'_1) = f^t(\gamma'_2)$, and which are maximal 
with respect to this property. We require furthermore that each $\gamma'_i$ 
contains all 
connecting paths 
of $\gamma_i$, 
and 
that $\gamma'_i$ 
overlaps non-trivially with the last edge of $\gamma_i$, 
allowing also 
the case $\gamma'_i = \gamma_i$.
In particular, each $\gamma'_i$ 
must be 
non-zero, 
$\gamma_i$ is the canonical vertex-prolongation of $\gamma'_i$, 
and the turn at $\chi$ must be illegal.

Then the set $\cal V(f^t)$ is finite.
\end{lem}

\begin{proof}
Let $\eta = \bar \gamma_1 \circ \chi \circ \gamma_2$ be an edge path in $\cal V(f^t)$.
By hypothesis every 
connecting path 
on 
either 
of the 
two
$\gamma_i$
belongs to $\gamma'_i$. From Lemma \ref{special-applied} we know that 
$\gamma'_1$ and $\gamma'_2$ 
are both 
pre-t-special. Since each $\gamma_i$ is the canonical vertex-prolongation of $\gamma'_i$, it follows that 
that both $\gamma_i$ are also pre-t-special.

Furthermore, 
the path $f^t(\bar \gamma'_1 \circ \chi \circ \gamma'_2)$ is a backtracking subpath of $f^t(\eta)$. Hence 
the length of each $\gamma_i$ is bounded, by 
Hypothesis \ref{running-hyp} (2)
and the 
convention that 
$|f(e)| \geq 1$ for any edge $e$ of $\cal G$.
Hence Lemma \ref{zero-finite} (2) shows that
there are only finitely many choices for $\gamma_1$ and $\gamma_2$. The hypothesis that the turn at $\chi$ is illegal thus proves 
by Lemma \ref{crucial-fin} (4) 
the finiteness of the set $\cal V(f^t)$.
\end{proof}

\begin{rem}
\label{V-sets}
Since the sets $\cal V(f^t)$ from Lemma \ref{first-finiteness} will play an important role in the next sections, we'd like to note right away some basic observations about them:
\begin{enumerate}
\item
The set $\cal V(f^t)$ contains every $f^t$-preimage turn of any degenerate turn. Conversely, every path $\eta = \bar \gamma_1 \circ \chi \circ \gamma_2$ in $\cal V(f^t)$ with lengths $|\gamma_1| = |\gamma_2| = 1$ must be the $f^t$-preimage turn of some degenerate turn.
\item
For any path $\eta = \bar \gamma_1 \circ \chi \circ \gamma_2$ from $\cal V(f^t)$ with $|\gamma_1| \geq 2$ or $|\gamma_2| \geq 2$ we 
write both $\gamma_i$ 
as legal concatenation $\gamma_i = \check \gamma_i \circ \chi_i \circ  e_i$, where $e_i$ is the last edge of $\gamma_i$ and $\check \gamma_i$ is its complementary initial subpath. (Here one of $\check \gamma_1$ or $\check \gamma_2$ could be zero, but not both.)

We consider the image paths $f^t(\gamma_1)$ and  $f^t(\gamma_2)$ and 
the corresponding 
legal concatenations $f^t(\gamma_i) = f^t(\check \gamma_i) \circ f^t(\chi_i) \circ  f^t(e_i)$, 
and 
we recall from the definition of $\cal V(f^t)$ that $f^t(\gamma_1)$ and  $f^t(\gamma_2)$ have a common initial subpath $\gamma_0$ which contains the initial subpaths $f^t(\check \gamma_1)$ and  $f^t(\check \gamma_2)$. Furthermore, the end of $\gamma_0$ must overlap 
non-trivially with both, the terminal subpath $f^t(e_1)$ of $f^t(\gamma_1)$ and the terminal subpath $f^t(e_2)$ of $f^t(\gamma_2)$.
We can thus distinguish the following three cases:
\begin{enumerate}
\item
If $f^t(\check \gamma_1)$ is strictly longer than $f^t(\check \gamma_2)$, then $f^t(\check \gamma_1)$  and $f^t(e_2)$ must have a 
non-zero overlap on 
the path $\gamma_0$, so that 
the path $\bar{\check\gamma}_1 \circ \chi \circ \gamma_2$ belongs to $\cal V(f^t)$.
\item
Similarly, if $f^t(\check \gamma_2)$ is strictly longer than $f^t(\check \gamma_1)$, then the path $\bar{\gamma}_1 \circ \chi \circ \check \gamma_2$ belongs to $\cal V(f^t)$.
\item
If $|f^t(\check \gamma_1)| = |f^t(\check \gamma_2)|$, then one has indeed $f^t(\check \gamma_1) = f^t(\check \gamma_2)$, so that $\bar{\check\gamma}_1 \circ \chi \circ \check \gamma_2$ belongs to $\cal V(f^t)$.
\end{enumerate}
\item
Consider any path $\eta = \bar \gamma_1 \circ \chi \circ \gamma_2$ in $\cal V(f^t)$, and let $\gamma'_1$ and $\gamma'_2$ be initial sub-edge-paths of the edge paths $\gamma_1$ and $\gamma_2$ respectively. Then the resulting path $\eta' = \bar \gamma'_1 \circ \chi \circ \gamma'_2$ is not necessarily an element of $\cal V(f^t)$. However, after comparing $|f(\gamma'_1)|$ to $|f(\gamma'_2)|$, one can 
take off from the $\gamma'_i$ with the longer $f^t$-image 
iteratively 
terminal edges (together with the connecting path at 
their 
beginning) until one finds an initial sub-edge-path $\gamma''_i$ of $\gamma'_i$ which has the property that the $f^t$-image of the terminal edge of $\gamma''_i$ is a subpath of $f^t(\gamma''_i)$ that overlaps 
non-trivially with the image path $f^t(\gamma'_j)$ of the ``other'' 
initial subpath $\gamma'_j$ (so that $\{i, j\} = \{1, 2\}$).

We set $\gamma''_j := \gamma'_j$ and obtain thus a 
subpath $\eta'' = \bar \gamma''_1 \circ \chi \circ \gamma''_2$ of the path $\eta' = \bar \gamma'_1 \circ \chi \circ \gamma'_2$. From the above construction we see that $\eta''$ agrees with $\eta'$ up to an initial or terminal sub-edge-path, and 
that $\eta''$ is an edge path canonically derived from $\eta'$ which 
does belong to $\cal V(f^t)$.

\item
Finally, we observe that one has $\cal V(f^t) \subset \cal V(f^{t'})$ for any $t' \geq t \geq 1$. This is indeed a direct consequence of the trivial observation that for any points $x_1$ on $\gamma_1$ and $x_2$ on $\gamma_2$ with $f^t(x_1) = f^t(x_2)$ one also has $f^{t'}(x_1) = f^{t'}(x_2)$.
Hence for any edge path $\eta = \bar \gamma_1 \circ \chi \circ \gamma_2$ in $\cal V(f^t)$ the maximal initial subpaths (``paths'' but not ``edge paths'' !) $\gamma'_1$ of $\gamma_1$ and $\gamma'_2$ of $\gamma_2$ with $f^t(\gamma'_1) = f^t(\gamma'_2)$, which are already considered in the statement of Lemma \ref{first-finiteness}, one also has $f^{t'}(\gamma'_1) = f^{t'}(\gamma'_2)$, which shows that all condition are satisfied to conclude that $\eta$ also belongs to $\cal V(f^{t'})$.
\end{enumerate}
\end{rem}

\section{INP paths in $\cal G$}
\label{sec:INP-paths}

\begin{convention}
\label{hyp-sec-4}
Throughout this section we assume that $\cal G$ is a graph-of-spaces and that $f: \cal G \to \cal G$ is an expanding train track map which satisfies the conditions (1) and (2) of Hypothesis \ref{running-hyp}.
\end{convention}

A 
reduced 
path 
as in Convention \ref{path} 
(but not necessarily an edge path \!\!\! !) $\eta =  \gamma \circ \chi \circ \gamma'$ in $\cal G$ is an {\em INP path} if both, $\gamma$ and $\gamma'$ are 
non-zero legal paths, 
and if for some exponent $t \geq 1$ the reduced path $[f^t(\eta)]$ is equal to $\eta$. In this case the 
turn used by $\eta$ at the concatenation vertex 
space $X_v$, which is 
defined by $\chi$ and the adjacent edges 
(or edge segments), 
is 
non-degenerate, since $\eta$ is assumed to be reduced.
But this turn is  
necessarily illegal, 
since 
we assume that $f$ is expanding,
so that the $f^t$-image turn at $X_{f^t(v)}$ must be degenerate. 
This ``concatenation vertex space'' which contains $\chi$ is called the {\em tip} of the INP path $\eta$, 
and the two maximal legal subpaths $\gamma$ and $\gamma'$ are the {\em legal branches} of $\eta\,$; 
we use this terminology also for other paths which are an illegal concatenation of two 
non-zero legal paths
through some connecting zero path.

For any INP path $\eta$ we consider 
the backtracking subpath $\check \eta$ of 
$f(\eta)$ 
which starts at the tip of $[f(\eta)]$, runs up to the tip of $f(\eta)$, and then doubles back to the tip of $[f(\eta)]$; in fact, $\check \eta$ is the maximal backtracking subpath of $f(\eta)$. The subpath $\eta_1 = \gamma_1 \circ \chi \circ \gamma'_1$ of $\eta$ which is mapped by $f$ to 
the 
backtracking subpath $\check \eta$ in $f(\eta)$ will be called {\em the $f$-backtracking subpath} of $\eta$. We denote the $f^t$-backtracking subpath of $\eta$ by $\eta_t = \gamma_t \circ \chi \circ \gamma'_t$, and 
observe that 
for any $t' \geq t \geq 1$ the path $\eta_{t'}$ contains $\eta_{t}$ as subpath.

\begin{rem}
\label{sub-edge-path}
Since it plays a role in what follows, we would like to point out that the above maximal backtracking subpath $\check \eta$ of $f(\eta)$ is an edge path: it starts and ends in an edge of $\cal G$. On the other hand, the ``preimage'' subpaths $\eta_t$ of $\eta$ are in general not edge paths, but they are ``paths'' 
in the sense of Convention \ref{path} 
in that they start and end with non-trivial edge segments.
\end{rem}

From the existence of a cancellation bound as in 
Hypothesis \ref{running-hyp} (2) 
we obtain 
exactly as for classical expanding train track maps (see for instance Remark 6.2 and Lemma 6.3 of \cite{CL}) 
the following:

\begin{lem}
\label{bounded-length}
There exists a constant $C_1 \geq 0$ which only depends on $\cal G$ and 
$f$, 
such that for any INP path $\eta$ in $\cal G$ the length of $\eta$ is bounded by 
$C_1$:  
one has
$$ |\eta_t | 
\leq  |\eta | \leq C_1$$
for any integer $t \geq 1$.
\qed
\end{lem}

We recall that 
an INP path $\eta$ 
is in general not an edge path: 
just as pointed out in Remark \ref{sub-edge-path} for the subpaths $\eta_t$, the path $\eta$ may well start or end at a point in the interior of an edge.
In order to ease notation, we will call any ``passage'' of $\eta$ through a vertex space a {\em vertex transition} of $\eta$. This includes any connecting path that occurs as subpath of $\eta$, as well as the initial or terminal point of $\eta$, if the latter do not lie in the interior of an edge of $\cal G$. In particular,
the two {\em extremal vertex transitions} of $\eta$, by which we mean the first and the last ``passage'' of $\eta$ through a vertex space, 
will in general not coincide with 
the endpoints of $\eta$. The same terminology will be used for any of the paths $\eta_t$.

\begin{lem}
\label{vertices-into}
There exists an exponent $t_1 \geq 1$ such that for any INP path $\eta$ all vertex transitions of $\eta$ except perhaps the two extremal ones 
are contained in the interior of 
the subpath 
$\eta_{t_1}$.
\end{lem}

\begin{proof}
This is a direct consequence of Lemma \ref{bounded-length}, since $f$ is assumed to be expanding.
In fact, it suffices to take $t_1$ big enough so that $|f^{t_1}(e)| > C_1$ for any edge $e$ of $\cal G$.
\end{proof}

A simplified version of the 
following is 
again 
well known 
for classical train track maps:

\begin{lem}
\label{vanishing}
(1)
For any point $x$ 
on an INP path $\eta$ which is not 
an endpoint of $\eta$ 
there is an exponent $t \geq 1$ such that 
$x$ is contained in the interior of the subpath $\eta_t$
(or equivalently: 
such that 
$f^{t}(x)$ is contained in the 
maximal 
backtracking path of $f^t(\eta)$ and is distinct from the tip of $[f^{t}(x)]$).  
Let $t(x) \geq 1$ be the smallest such exponent $t$.

\smallskip
\noindent
(2)
Assume that the point $x$ on $\eta$ is a boundary point of an edge $e$ of $\cal G$, and that $e$ is contained in the subpath that connects $x$ to
the tip of $\eta$.
Then there is a unique  point 
$x'$ on the other legal branch of $\eta$ than $x$ which satisfies:
\begin{enumerate}
\item[(a)]
$f^{t(x)}(x') = f^{t(x)}(x)$,
\item[(b)]
$x'$ is a point on some edge $e'$ of $\cal G$, and
\item[(c)]
some non-trivial segment of $e'$ is contained in the subpath that connects $x'$ to the tip of $\eta$.
\end{enumerate}
Furthermore, both $x'$ and $x$ are contained in the interior of the $f^{t(x)}$-backtracking subpath $\eta_{t(x)}$ of~$\eta$.
\end{lem}

\begin{proof}
Just as for classical expanding train track maps,
statement (1) is equivalent to the observation that $\eta$ is equal to the closure of the union of all $\eta_t$, which is an easy consequence of the definition of an INP path, as $f$ is assumed to be expanding.

\smallskip

For statement (2) we recall for any $t \geq 1$ the decomposition $\eta_t = \gamma_t \circ\chi_t \circ \gamma'_t$ with legal branches $\gamma_t$ and $\gamma'_t$, so that the image paths $f^t(\gamma_t)$ and $f^t(\gamma'_t)$ are identical edge paths, up to (as usual) homotopy within vertex spaces. It follows that for any point $x$ on $\eta$ there exists a point $x'$ which lies on the other legal branch of $\eta$ than $x$, and satisfies $f^t(x) = f^t(x')$, but a priori only up to homotopy within some vertex space of $\cal G$.
In order to simplify notation we will from now on assume that $x$ lies on $\gamma$ and $x'$ on $\gamma'$.

The point 
$x$ is assumed to be the boundary point of some edge $e$ in $\cal G$, 
and 
$f^{t(x)}(e)$ is an edge path and thus starts and ends in an edge. 
It follows that $f^{t(x)}(x)$ is the boundary point of an edge, 
and is hence 
contained in some vertex space $X_v$. 
If the above point $x'$ is contained in the interior of an edge, then one has $f^{t(x)}(x') = f^{t(x)}(x)$, and this equality determines $x'$ uniquely, since within each legal branch there is no 
non-zero backtracking path.
Clearly the above conditions (b) and (c) are also satisfied by $x'$.

If the above point $x'$ is not contained in the interior of an edge, 
it could be a priori any point in a vertex space $X_{v'}$ 
traversed by $\gamma'$, such that $X_{v'}$ 
is mapped by $f^{t(x)}$ to $X_v$.
However, since vertex spaces are mapped by $f$ to vertex spaces, there must be an edge $e'$ on $\gamma'$ with terminal  endpoint in $X_{v'}$, such that the edge paths $f^{t(x)}(e)$ and $f^{t(x)}(\bar e')$ have the same initial edge.
Hence setting $x'$ to be the terminal endpoint of $e'$ satisfies conditions (a) - (c). The only other point in $X_{v'}$ which also satisfies 
the 
conditions (a) and (b) is the initial point 
$x''$ 
of the edge $e''$ which succeeds $e'$ on the path $\gamma'$, but this point 
$x''$ 
does not satisfy condition (c), as $e''$ 
meets 
the subpath of $\eta$ bounded by $x$ and $x''$ only in 
the 
point $x''$.

The last sentence in the statement (2) is a direct consequence of the equivalence spelled out in the parenthesis in statement (1), since by assumption $f^{t(x)}(x)$ is contained in the backtracking path of $f^{t(x)}(\eta)$ and is distinct from the tip of $[f^{{t(x)}}(x)]$.
\end{proof}

\begin{rem}
\label{purity}
In the above proof we have shown that the following more general 
version of 
Lemma \ref{vanishing} (2) 
also holds:

\smallskip
{\em
For any point $x$ on 
an INP path 
$\eta$ 
there is a point 
$x' \in \eta$, 
which lies on the other legal branch of $\eta$ than $x$, 
such that, up to homotopy within  vertex spaces, one has $f^{t(x)}(x') = f^{t(x)}(x)$.
Both, $x'$ and $x$ are contained in the interior of the $f^{t(x)}$-backtracking subpath $\eta_{t(x)}$ of $\eta$.}

\smallskip
\noindent
The reason for imposing in Lemma \ref{vanishing} (2) the extra conditions 
(a) - (c) 
on $x$ and $x'$ is the following: These points bound a subpath $\beta$ of $\eta_{t(x)}$, which plays a vital role in the rest of this section. 
The conditions (a) - (c) ensure that 
the term ``subpath'' is justified for $\beta$, as 
these 
conditions guarantee that $\beta$ 
starts and ends with non-trivial edge segments, so that $\beta$ 
is indeed a ``path'' in the strict meaning 
of Convention \ref{path}. 
Without this rigid restriction and our stubborn insistence on it in the previous sections, the finiteness arguments needed 
below would be fudged.
\end{rem}

If for some path $\gamma$ in $\cal G$ 
a subpath $\gamma_0$ of $\gamma$ is mapped by $f$ to a backtracking subpath of $f(\gamma)$, 
then the same is true for any other path $\gamma'$ which contains $\gamma_0$ as subpath. In Lemma \ref{vanishing} (2), however, the subpath 
$\beta$ 
of $\eta$ bounded by $x$ and $x'$ is not just contained in the $f^{t(x)}$-backtracking subpath $\eta_{t(x)}$ of $\eta$, but it is required to lie in the {\em interior} of $\eta_{t(x)}$. For any classical train track map, i.e. if all vertex groups of $\cal G$ are trivial, this subtle difference is immaterial. In our context, however, we need to consider this question in detail.

\begin{lem}
\label{in-detail}
Let $\eta, x$ and $x'$ be as in Lemma \ref{vanishing} (2). Then there exists an exponent $\hat t(x) \geq 1$, which only depends on the subpath $\beta$ of $\eta$ that is bounded by $x$ and $x'$ 
- but not on the path $\eta$ itself,  such that $\beta$ (and hence $x$ and $x'$) are contained in the the interior of the $f^{\hat t(x)}$-backtracking subpath $\eta_{\hat t(x)}$ of $\eta$.
\end{lem}

\begin{proof}
We first consider 
the exponent 
$t := t(x)$ as in Lemma \ref{vanishing} (1). Thus we have $f^t(x) = f^t(x')$, and $\beta$ is contained in the $f^t$-backtracking subpath $\eta_t$ of $\eta$. 
For any second INP path $\eta'$, which also contains $\beta$ as subpath, this subpath must also be contained in the $f^t$-backtracking subpath $\eta'_t$ of $\eta'$.

If the point $f^t(x) = f^t(x')$ is 
contained in the interior of an edge, 
then the boundary points of $\beta$ must both be distinct from the boundary points of $\eta'_t$, so that in this case we can set $\hat t(x) = t$.

If $f^t(x) = f^t(x')$ is 
contained in a vertex space, 
then we consider the path $f^t(\eta')_0$, obtained from $f^t(\eta')$ by reducing the backtracking path 
$f^t(\beta)$: 
we have $f^t(\eta')_0 = : \gamma_0 \circ \chi_0 \circ \gamma'_0$, where the legal paths $\gamma_0$ and $\gamma'_0$ are 
a 
non-zero 
initial and 
a non-zero 
terminal subpath of the 
$f^t$-images of the 
two maximal legal subpaths of $\eta'$, and $\chi_0$ is a connecting path in the vertex space which contains $f^t(x) = f^t(x')$.

In this case the path 
$f^t(\eta')_0$ 
defines a (possibly degenerate) illegal turn $e \circ \chi_0 \circ e'$ at $\chi_0$, where $e$ is the last edge of $\gamma_0$ and $e'$ is the first edge of $\gamma'_0$ (or 
of their canonical vertex-prolongations). 
Hence we can apply Lemma \ref{crucial-fin} (5)
to obtain an exponent $t_0 \geq 0$ which only depends on $f$ and $\cal G$ and not on the particular choice of $e, e'$ and $\chi_0$, such that the $f^{t_0}$-image turn of $e \circ \chi_0 \circ e'$ is degenerate. It follows that for any choice of $\eta'$ as above the subpath $\beta$ must be contained in the $f^{t+t_0}$-backtracking subpath of $\eta'$.

Hence setting $\hat t(x) = t + t_0$ satisfies our claim in both cases.
\end{proof}

For the following result it is important to recall that in this paper an INP path in $\cal G$ is always a non-zero path as in Convention \ref{path}. In particular, it starts and ends always in a non-degenerate segment of a graph-of-spaces edge from $\cal G$ (but its endpoints may well lie in the interior of such an edge). 

\begin{prop}
\label{INP-finite}
Let $\cal G$ and $f: \cal G \to \cal G$ be as in Convention \ref{hyp-sec-4}.
Then there 
are only finitely many INP paths in $\cal G$.
Indeed, there exists an integer $\hat t \geq 1$ such that every INP path is a subpath of 
some edge path in the finite set $\cal V(f^{\hat t})$.
\end{prop}

\begin{proof}
For $t_1$ as in Lemma \ref{vertices-into} we apply Lemma \ref{first-finiteness} to conclude that the 
edge path 
set $\cal V(f^{t_1})$ is finite. 
From the definition of $t_1$ in Lemma \ref{vertices-into} it follows that for any INP path $\eta$ the subpath $\eta_{t_1}$ contains all 
vertex transitions of $\eta$ except possibly the two extremal ones, and that $\eta_{t_1}$ doesn't start or end in any of the two 
vertex transitions 
on $\eta$ that are next to the extremal ones. It follows that the canonical vertex-prolongation $\eta'$ of $\eta_{t_1}$, which is
the subpath of $\eta$ that is bounded by the two extremal 
vertex transitions 
of $\eta$, is contained in $\cal V(f^{t_1})$.

For 
any INP path $\eta$ 
and its sub-edge-path $\eta'$ in $\cal V(f^{t_1})$ as above 
we consider the 
initial and terminal endpoints 
$x_1$ and $x_2$ 
of $\eta'$, 
which lie in the extremal vertex transitions of 
$\eta$. 
Let $x'_1$ and $x'_2$ be the corresponding points on $\eta$ as provided by Lemma \ref{vanishing} (2), and let $\beta_1$ and $\beta_2$ be the subpaths of $\eta$ bounded by $x_1$ and $x'_1$, or by $x_2$ and $x'_2$ respectively. Without loss of generality we can assume that $\beta_1$ is a subpath of $\beta_2$, and while $\beta_1$ is necessarily a subpath of $\eta'$, we observe that the analogue statement for $\beta_2$ fails unless one has $\beta_1 = \beta_2$ (and hence $x_1 = x'_2$ and $x_2 = x'_1$).

In case that $x_1$ is distinct from the endpoint $x'_2$ of $\beta_2$, we apply Lemma \ref{in-detail} to obtain an exponent $t_2 : = \hat t(x_1) \geq 1$. Otherwise we pose $t_2 = t_1$, and notice that in either case the canonical vertex-prolongation $\hat \beta_2$ 
of $\beta_2$ is contained in $\cal V(f^{t_2})$ (which is also finite, again by Lemma \ref{first-finiteness}). 

We now 
``repeat'' 
the same procedure with $x_2$: 
If $x_2$ is not an endpoint of $\eta$, we apply Lemma \ref{in-detail} to obtain an exponent $t_3 : = \hat t(x_2) \geq 1$. Otherwise we pose $t_3 = t_2$ and observe that 
now the canonical vertex-prolongation $\hat \eta$ 
of $\eta$ is contained in $\cal V(f^{t_3})$.

We recall that the edge path $\beta_1$ is a subpath of one of the edge paths $\eta'$ from the finite set $\cal V(f^{t_1})$. The exponent $t_2$ only depends on one of the endpoints of $\eta'$, and  
the edge path $\beta_2$ is a subpath of one of the edge paths $\hat \beta_2$ from the finite set $\cal V(f^{t_2})$. Finally, the exponent $t_3$ only depends on one of the endpoints of $\hat \beta_2$, so that 
as consequence we deduce that 
the total set of any such $t_3$ is finite. It follows that among them there exists a maximal exponent $\hat t \geq 1$, 
and that every INP path $\eta$ is a subpath of some edge path $\hat \eta$ from $\cal V(f^{\hat t})$, which (again by Lemma \ref{first-finiteness}) is also 
a finite 
set. 

We now apply 
Lemma \ref{vanishing} 
(1) to 
deduce 
that no path with a single illegal turn can contain two distinct INP paths as subpaths, and thus obtain the desired finiteness of all INP paths in $\cal G$.
\end{proof}

\begin{rem}
\label{Collins-Turner}
In the context of generalizing Bestvina-Handel's work on relative train track representatives from free groups to free products, Collins and Turner considered in \cite{CT} already indivisible Nielsen paths in ``graphs of complexes'' and ``topological maps'', which seem close to what is called here ``graph-of-space morphisms''. In Proposition 2.4 of \cite{CL} a finiteness statement for ``real ended INPs'' is proved. However, unless we misunderstand their definitions, their result does not concern INP paths that are periodic (up to cancellation at the tip) but not fixed, and the main difficulty, that such periodic INP paths do in general not start at a vertex, seems to be circumvented a priori by their definition of a ``real ended INP''.
\end{rem}

\section{INP candidates in $\cal G$}
\label{sec:INP-candidates}

In order to prove statements (2) and (3) from Theorem \ref{main} we need to reconsider some of the material from the previous section, but for slightly more general than just INP paths: In this section we denote by $\eta = \gamma \circ \chi \circ \gamma'$ any concatenation of two non-zero legal paths 
(in the strict meaning of Convention \ref{path}) 
$\gamma$ and $\gamma'$, and 
we 
assume that the turn at the concatenation vertex space (which contains 
the connecting zero path $\chi$) is illegal. Such a path will be called a {\em pseudo-INP path}. The goal of this section is to show that any sufficiently large $f$-iterate of a pseudo-INP path becomes after reduction either a legal path, or 
else 
this reduced iterate 
will contain 
an honest INP path. Furthermore, there is an upper bound to the number of 
iterations 
needed 
here.

For any exponent $t \geq 1$ we define the $f^t$-backtracking subpath $\eta_t$ of $\eta$ precisely as 
done (before Remark \ref{sub-edge-path}) for an INP path, and we 
define $\eta_\infty$ as the closure of the union of all $\eta_t$. 
The canonical vertex-prolongation of $\eta_\infty$ will be denoted by $\hat \eta$.
We 
observe:

\begin{rem}
\label{pseudo-INP-1}
Let $\eta = \gamma \circ \chi \circ \gamma'$ be a pseudo-INP path in $\cal G$. Then the following holds:
\begin{enumerate}
\item
For any 
exponent $t \geq 1$ 
the reduction $[f^t(\eta)]$ of $f^t(\eta)$ is legal if and only if one has
$\eta_{t'} = \eta_{t}$ for 
all 
$t' \geq t$.

The smallest such exponent $t$ is denoted by $t(\eta)$ and is called the {\em legalizing exponent} for~$\eta$.
\item
The path $\eta$ possesses a legalizing exponent $t(\eta) \in \N$ if and only if 
one of the following occurs:
\begin{enumerate}
\item
At least 
one of $f^{t(\eta)}(\gamma)$ or $f^{t(\eta)}(\gamma')$ is completely contained in the backtracking subpath of $f^{t(\eta)}(\eta)$. This is equivalent to stating that 
at least 
one of $\gamma$ or $\gamma'$ is completely contained in the subpath $\eta_{t(\eta)}$.
\item
Neither $f^{t(\eta)}(\gamma)$ nor $f^{t(\eta)}(\gamma')$ is completely contained in the backtracking subpath of $f^{t(\eta)}(\eta)$, and the turn in $[f^{t(\eta)}(\eta)]$ at its ``tip'' (by which we mean 
the vertex space that contains the boundary points of the maximal backtracking subpath 
in 
$f^{t(\eta)}(\eta)$) is legal. This turn is the only turn on the legal path $[f^{t(\eta)}(\eta)]$ which is (possibly) not used by either of the legal paths $f^{t(\eta)}(\gamma)$ or $f^{t(\eta)}(\gamma')$.
\end{enumerate}
\end{enumerate} 
\end{rem}

We now 
observe that any 
honest INP path 
$\eta$ is in particular a pseudo-INP path, with the property that $\eta = \eta_\infty$. For our purposes it suffices 
to consider only {\em short pseudo-INP edge paths},
by which we mean any pseudo-INP path 
$\eta$ 
which is equal to the canonical vertex-prolongation 
$\hat \eta$ 
of its subpath $\eta_\infty$.

\begin{prop}
\label{short-INP}
Every short pseudo-INP edge path 
$\eta = \hat \eta$ 
is contained in the finite set $\cal V(f^{\hat t})$ from Proposition \ref{INP-finite}.
\end{prop}

\begin{proof}
The proof is already given in the last section, by Lemmas \ref{bounded-length}
-- \ref{in-detail} and by Proposition \ref{INP-finite}, if one uses the following ``translation'' in the reading of these statements: Any time an INP path $\eta$ is invoked, is has to be replaced by the subpath $\eta_\infty$ of any 
pseudo-INP path. 
With this ``translation'' the proofs of these statements, as stated in the previous section, are true word-by-word. In fact, they are written already with this more general application in mind. (To be meticulous, we should add here that the assertion at the beginning of the next-to-last paragraph in the proof of Lemma \ref{in-detail}, that the turn $e \circ \chi_0 \circ e'$ is illegal, is justified in the above ``translation'' by the assumption - inherited from Lemma \ref{vanishing} - that $x$ is not an endpoint of $\eta_\infty$). 
In particular, in the proof of Proposition \ref{INP-finite} it is shown that $\hat \eta$ belongs to $\cal V(f^{\hat t})$.
\end{proof}

For any edge path $\hat \eta$ in $\cal V(f^{\hat t})$ we now want to consider 
an ``image edge path'' by proceeding as follows: We (i) first apply $f$ to get $f(\hat \eta)$, (ii) next reduce $f(\hat \eta)$ to get $[f(\hat\eta)]$, 
then (iii) (unless $[f(\hat\eta)]$ is legal) consider the subpath $[f(\hat\eta)]_\infty$ of $[f(\hat\eta)]$, 
and (iv) finally pass to the canonical vertex-prolongation of $[f(\hat\eta)]_\infty$. The resulting edge path
$\hat{[f(\hat \eta)]}$
is again a well defined short pseudo-INP edge path and thus contained in $\cal V(f^{\hat t})$, by Proposition \ref{short-INP}.

We now define $\cal V_+$ to be obtained from $\cal V(f^{\hat t})$ by adding a formal symbol $*$, and we set $\hat f(\hat \eta) = *$ if $[f(\hat \eta)]$ is legal, and $\hat f(\hat \eta) = \hat{[f(\hat \eta)]}$ otherwise. We thus obtain a map 
$$\hat f: \cal V_+ \to \cal V_+$$
by completing the above 
settings through postulating 
$\hat f(*) = *$. 
From 
this 
definition we verify immediately that 
the map $\hat f$ 
satisfies $\hat f(\hat \eta) = \hat{f(\eta)}$ for any pseudo-INP path $\eta$.

\begin{cor}
\label{iteration-bound}
The cardinality of $\cal V(f^{\hat t})$ is an upper bound to the number of iterations of the map $f$ needed to be performed on any pseudo-INP path $\eta$, in order to obtain (after reduction) either a legal path, or a path that contains an honest INP path as subpath.
\end{cor}

\begin{proof}
It suffices to consider for any pseudo-INP path $\eta$ the associated short pseudo-INP edge path $\hat \eta \in \cal V(f^{\hat t})$, 
as well as the 
above map $\hat f: \cal V_+ \to \cal V_+$. 
From the definition of an INP-path it follows directly that any path $\hat \eta$ in the set $\cal V(f^{\hat t}) \subset \cal V_+$, which we assume to have a periodic $\hat f$-orbit, must 
contain an INP-path as subpath. This proves our claim, since the only other periodic orbit in the finite set $\cal V_+$ consists of the point $*$ only, and for any path $\hat \eta$ in  $\cal V(f^{\hat t})$ which is mapped by $\hat f$ to $*$ 
the reduced $f$-image $[f(\hat \eta)]$ is legal
(by the above definition of $\hat f$).
\end{proof}

The statement of Corollary \ref{iteration-bound} is all one needs to derive the following Proposition \ref{from-LUbig}, which extends the important well known fact about the behavior of illegal turns under iteration of $f$ from classical train track maps to graph-of-groups train track maps. In fact, 
its proof uses the very same arguments as applied previously in more than one occasion for classical train track maps, see for instance \cite{LL3}, Lemma 6.1. or 
\cite{LU-big}, Propositions 4.12 and 4.18, so that we can leave the ``translation'' to the interested reader.
We recommend in particular section 4.2 of \cite{LU-big}, where a detailed exposition of these arguments can be found, in a terminology not far from the one used here. 

\begin{prop}
\label{from-LUbig}
Let $f: \cal G \to \cal G$ be an expanding train track 
map
that satisfies conditions (1) and (2) of Hypothesis \ref{running-hyp}. 
Then the following holds:
\begin{enumerate}
\item
There exists an exponent 
$t = t(f) \geq 1$ such that for any edge path $\gamma$ in $\cal G$ the number of illegal turns in the reduced path $[f^t(\gamma]$ is strictly smaller than those in $\gamma$, unless $[f^t(\gamma)]$ is a legal concatenation of legal and INP paths.
\item
For every edge path 
(or loop) 
$\gamma$ in $\cal G$ there is an integer 
$t(\gamma) \geq 0$ such that the reduced path $[f^{t(\gamma)}(\gamma)]$ derived from $f^{t(\gamma)}(\gamma)$ is a 
(cyclically) 
legal concatenation of 
legal and INP subpaths. 
An upper bound 
$\hat t(\gamma)$ 
for the exponent $t(\gamma)$ depends only on the number of illegal turns in $\gamma$ and not on the particular choice of $\gamma$ itself. 
\item
There is an exponent 
$t_4 \geq 1$ 
such that for any edge path or loop $\gamma$ one has the following inequality, where $ILT(\gamma)$ denotes the number of illegal turns in $\gamma$:
\begin{equation}
\label{expo-dec}
ILT(f^{t_4}
(\gamma)) - ILT([f^{t(\gamma)}(\gamma)]) \leq \frac{1}{2} \left( ILT(\gamma) - ILT([f^{t(\gamma)}(\gamma)]) \right)
\end{equation}
\end{enumerate}
\qed
\end{prop}

We end this section by 
summing up 
the results derived 
in this and in the previous section 
in the following proof. 
We recall (see Lemma \ref{C-bound}) that any $\cal G$ and $f$ as in Theorem \ref{main} satisfy conditions (1) and (2) of Hypothesis \ref{running-hyp}.

\begin{proof}[Proof of Theorem \ref{main}]
Statement (1) of Theorem \ref{main} has been proved in Proposition \ref{INP-finite}. Statement (2) is part of Proposition \ref{from-LUbig} (2). Finally, statement (3) is the special case of Proposition \ref{from-LUbig} (3) for legal $\gamma$.
\end{proof}

\section{Algorithms}
\label{sec:algorithm}

For many algorithmic question about automorphisms of free groups (or similar groups such as certain free products or mapping class groups) the efficient determinations of all INP paths in a 
given 
train track map is a crucial task. In the context of graph-of-space train track maps as considered here, all interesting data (like the exponents
in Proposition \ref{from-LUbig}) can be readily computed, once the set $\cal V(f^{\hat t})$ from Proposition \ref{INP-finite} has been determined. 
For this purpose the following needs to be satisfied:

\begin{hypothesis}
\label{given-data}
We assume that $f: \cal G \to \cal G$ is 
an expanding train track map. In particular, 
any vertex space $X_v$ is mapped to a vertex space 
$X_{f(v)}$, and any edge $e$ is mapped to an edge path $f(e)$. We assume furthermore:
\begin{enumerate}
\item
The map $f$ possesses a cancellation bound $C \geq 0$ as in Hypothesis \ref{running-hyp} (2).
\item
For any vertex space $X_v$ the induced map $\pi_1 X_v \to \pi_1 X_{f(v)}$ is explicitly given in algorithmic terms.
\item
For any connecting path $\chi$ in any vertex space $X_v$ the set of all preimage paths 
of $\chi$  
(up to homotopy relative endpoints)
in any 
vertex space $X_{v'}$ with $f(v') = v$ must be efficiently computable. 
\end{enumerate}
\end{hypothesis}

\begin{rem}
\label{algo-hyp}
(1) To be specific let us record the following: we only require in Hypothesis \ref{given-data} (1) that a cancellation bound $C \geq 0$ exists. This knowledge is needed to make sure that the algorithm presented below stops 
eventually, while the bound itself  is then 
calculated by the 
algorithm. In the case where 
a cancellation bound 
is known beforehand, 
the given algorithm can be streamlined slightly.

\smallskip
\noindent
(2) In the special case, where $\pi_1 \cal G$ is a free group 
$\FN$ of finite rank 
$N \geq 0$, 
and 
where 
for each vertex $v$ of $\cal G$ the induced isomorphism 
$G_v \to G_{f(v)}$ is given by finitely many data, then both conditions~(2) and (3) of Hypothesis \ref{given-data} are satisfied.

This is true in particular 
if each vertex group $G_v$ is topologically realized by a finite graph $X_v$, and if $f$ is a combinatorial map of the resulting total graph $\Gamma$ with 
$\pi_1 \Gamma = \FN$  (where $\Gamma$ consists of all local edges of any $X_v$, and of all graph-of-groups edges of $\cal G$).
\end{rem}

We will now list the steps which have to be carried out in order determine first the set $\cal V(f^{\hat f})$ and then the other data stated in the previous two sections:

\medskip
\noindent
STEP 1:
For each of the finitely many degenerate turns in $\cal G$ we use Hypothesis \ref{given-data} (3) to compute the 
complete list of its preimage turns, which by Lemma \ref{crucial-fin} (3) is finite.
Hence 
one can re-iterate the procedure and compute the finite list of turns that are $f^2$--preimage turns of any degenerate turn. We now pass to the $f^3$-preimage turns, and so on, and note that according to Lemma \ref{crucial-fin} (2) no turn can show up twice on our lists. Hence, according to Lemma \ref{crucial-fin} (4), after finitely many iterations we will find an exponent $t \geq 1$ such that the list of $f^t$-preimage turns of any degenerate turn is empty.  Hence, through setting $t_0 = t -1$, we have computed the exponent $t_0$ from Lemma \ref{crucial-fin} (5).

We also note that the total list of all turns computed this way, or in other words, the union of all the various intermediate lists, including all of the original degenerated turns we started out with, is identical to the list of all illegal turns in $\cal G$. This is a direct consequence of the definition of an illegal turn in Definition-Remark \ref{legal+} (3).

\medskip
\noindent
STEP 2:
We compute the finite list of 1-special turns as given in Definition \ref{special-words} (1), and use Hypothesis \ref{given-data} (3) to determine the list of their preimage turns, which by Lemma \ref{crucial-fin} (3) is also finite. We eliminate from this list any turn that shows up in the total list from STEP (1), since we are only interested in legal turns. Armed with this list of ``allowed turns'' and the knowledge (from Lemma \ref{special-applied}) that the legal branches $\gamma_1$ and $\gamma_2$ of any path $\eta = \bar \gamma_1 \circ \chi \circ \gamma_2$ from $\cal V(f)$ can only use these ``allowed turns'', we can produce by trial and error the list of all paths from $\cal V(f)$ as follows: 
We start out with all preimage turns of any degenerate turn, and keep increasing our lists by iterating the following procedure: we consider, for any path 
$\eta = \bar \gamma_1 \circ \chi \circ \gamma_2$ already in the list, the path $\eta'$ which derives from $\eta$ by adding on to the end of $\gamma_1$ and/or to the end of $\gamma_2$ a single edge, by means of a connecting turn that is ``allowed''. One now applies $f$ to $\eta'$ and checks 
whether $\eta'$ belongs to $\cal V(f)$ (in which we add it to our list), or not (in which case $\eta'$ is eliminated from further considerations).

Our list of paths keeps increasing, but at any time there are only finitely many possibilities for paths $\eta'$ as just described. Since we know from Lemma \ref{first-finiteness} that $\cal V(f)$ is finite, this iterative procedure will eventually stop. It follows that the thus computed list is equal to $\cal V(f)$, since 
we know from Remark \ref{V-sets} (2) that 
any path $\eta$ from $\cal V(f)$ can reduced by the inverse procedure, i.e. taking off iteratively terminal edges from either $\gamma_1$ and/or $\gamma_2$, to some preimage turn of a degenerate turn, such that at any intermediate step one has a path from $\cal V(f)$.

\medskip
\noindent
STEP 3:
From the knowledge of $\cal V(f)$ we compute now a cancellation constant $C \geq 0$ as maximal length of the paths from $\cal V(f)$. This is not quite the bound from Hypothesis \ref{running-hyp} (2), since we do not consider all reduced paths, but only reduced products of two legal paths; however, cancellations at images of such paths are the only ones that are relevant for our arguments.

From this constant $C$ we next compute the constant $C_1 \geq 0$ from Lemma \ref{bounded-length}, by using the formula 
given in Lemma 6.3 of \cite{CL}, translated into our setting through $C(f)^+ = \frac{C_1}{2}, K = 1, C(f) = \frac{C'}{2}$ and $\lambda_{min}^K = \lambda_{min}$, to obtain:
$$
C_1 = \frac{C'}{\lambda_{min} - 1}
$$
Here the constants $C'$ and $\lambda_{min}$ are calculated as follows: We first use the hypothesis that $f$ is expanding to calculate an exponent $t \geq 1$ such that $|f^t(e)| \geq 2$ for any edge $e$ of $\cal G$, and set $\lambda_{min}$ and $\lambda_{max}$ to be the minimum and the maximum respectively of the edge image lengths $|f^t(e)|$. We then compute a cancellation bound $C'$ for $f^t$ from $C$ through 
$$
C' \,\,= \,\,\sum_{k = 1}^t C \lambda_{max}^{k-1} \,\,=\,\, C \, \frac{\lambda_{max}^t-1}{\lambda_{max}-1}
$$
and recall the observation before Remark \ref{sub-edge-path}, that for any exponents $t' \geq t$ the path $\eta_t$ is a subpath of $\eta_{t'}$. This shows that an upper bound for length of the path $\eta_\infty$ for any pseudo-INP path $\eta$, as defined at the beginning of section \ref{sec:INP-candidates}, computed for any positive power of $f$, is also valid for $f$.

In order to complete STEP 3 of our algorithm we now have to calculate the exponent $t_1$ from Lemma \ref{vertices-into} by iterating $f$ on the edges $e$ of $\cal G$ until one has $|f^{t_1}(e)| \geq C_1$ for any edge $e$.

\medskip
\noindent
STEP 4:  We now consider the proof of Lemma \ref{in-detail} and notice that for every path $\eta$ in $\cal V(f^{t_1})$ any point $x$ on any vertex transition of $\eta$ which is distinct from the two boundary points of $\eta$ is folded by $f^{t_1}$ onto some point $x'$ on the other legal branch of $\eta$ than $x$. It follows that the exponent $t_1$ is an upper bound to the value of the exponent $t(x)$ from Lemma \ref{vanishing} (1). Following through till the end of the proof of Lemma \ref{in-detail} yields thus that the exponent $\hat t(x)$ is bounded above (independently of the above choices for $\eta$ and $x$) by the sum $t_1 + t_0$, and both of these terms have been computed in the previous steps. 

We can thus pass to the proof of Proposition \ref{INP-finite}, where the value $\hat t$ is defined 
through finitely many choices for exponents $t_2$ and $t_3$. Here any $t_2$ is equal to $t_1$ or to $\hat t(x_1)$ for some point $x_1$ as the point $x$ above, so that $t_2$ is bounded above by $t_1 + t_0$. The value of any $t_3$ in turn is equal to some $t_2$ or to $t_2 + t(x_2)$, where $x_2$ is another point as $x$ above, only with $t_2$ in place of $t_1$. We thus obtain $t_1 + 2 t_0$ as upper bound for any of the finitely many values of $t_3$, and thus also for $\hat t$, which is defined as the maximum of all $t_3$.

Since from Remark \ref{V-sets} (4) we know that $\cal V(f^t) \subset \cal V(f^{t'})$ for any $t' \geq t \geq 1$, it suffices now to compute $\cal V(f^{t_1 + 2t_0})$. This can be done just as the computation of $\cal V(f)$ in STEP 2. Alternatively on can compute an upper bound for the length of any $\eta \in \cal V(f^{t_1 + 2t_0})$ just as in STEP 3, calculate the finite set of all pre-$(t_1+2t_0)$-special turns as in STEP 2, and do a more direct trial-an-error search among all edge paths defined by this bound, 
and by the set of ``allowed'' turns as in STEP 2.

\medskip
\noindent
STEP 5:  One now passes to the set $\cal V_+$ and the map $\hat f_+ :
\cal V_+ \to \cal V_+$ as defined in section \ref{sec:INP-candidates}, for $\cal V(f^{t_1 + 2t_0})$ in place of $\cal V(f^{\hat t})$. For any $f_+$-periodic edge path $\eta$ in $\cal V_+$ we now subdivide the first and the last edge by introducing new inessential vertices, so that the INP-sub-path $\eta_\infty$ of $\eta$ becomes an honest edge path.

This gives the possibility to calculate the maximal number $t_*$ of iterations of $f$ need to apply to any path $\alpha$ which is the union of two paths from $\cal V_+$ with 
non-zero overlap, in order to achieve (following Lemma \ref{vanishing} (1)) that $[f^{t_*}(\alpha)]$ is legal.
Furthermore, 
one needs to compute the maximal seize $t_+$ of any $\hat f_+$-orbit.
It follows then from the considerations in subsection 4.2 of \cite{LU-big} that $t_* + t_+$ is an upper bound to the exponent $t_4$ in Proposition \ref{from-LUbig} (3).

\medskip

As a consequence we now have:

\begin{proof}[Proof of Theorem \ref{algo-det}]
We first recall (see 
Lemma \ref{C-bound}) 
that the assumption, that $f$ is an expanding train track map which induces an automorphism of the free group $\pi_1 \cal G$, implies that all parts of Hypeothesis \ref{running-hyp} are satisfied. Hence the results of sections \ref{sec:INP-paths} and \ref{sec:INP-candidates} apply, and in particular 
(see Remark \ref{algo-hyp} (2)) 
the above 5 steps of our algorithm can be performed. From the finite set $\cal V_+$ we obtain the set of all INP paths in $\cal G$ as subset of all $\hat f_+$-periodic paths, which gives assertion (1) of Theorem \ref{algo-det}. Part (3) is achieved in STEP 5 above through setting $\hat t = t_4$. Finally, for any edge path or loop $\gamma$ in $\cal G$ we 
count 
the number of illegal turns and apply the formula (\ref{expo-dec}) to compute the exponent $t(\gamma)$ from part (2).
\end{proof}

The  computation of the set of INP-paths (and also of the above exponents $\hat t$ and $t(\gamma)$) for expanding train track maps of graphs-of-spaces are useful for many algorithmic purposes. In particular they play a crucial role in the construction of a $\beta$-train track representative for any automorphism $\phi \in \Out(\FN)$
(see \cite{Lu_conj-pr1_MPI, GautLu, Lu-PF, Lu-pre}
and also Remark \ref{beta-special} below), which was the original motivation for this paper. Here we restrict ourselves to the observation that the following corollary can be used to determine generators for the subgroup of elements of $\pi_1 \cal G$ that are element-wise fixed, once the analogous algorithmic task is solved for each of the induced vertex group automorphisms.

\begin{cor}
\label{fixed-conj}
Let $f: \cal G \to \cal G$ be as in Proposition \ref{from-LUbig}. 
Let $[g]$ denote a conjugacy class in $\pi_1 \cal G$
which is not represented by any loop in one of the vertex spaces.

Then $[g]$ is fixed by 
$f_*: \pi_1 \cal G \to \pi_1 \cal G$ if and only if $[g]$ is represented by a loop 
in $\cal G$ which is a 
cyclically legal 
concatenation (see Definition-Remark \ref{legal+} (6))
\begin{equation}
\label{with-s}
\gamma = \eta_1 \circ \chi_1 \circ \ldots \eta_q \circ \chi_q \qquad \text{(with $q \geq 1$)}
\end{equation}
of INP paths $\eta_i$ and zero paths $\chi_i$, such that for some integer $s \geq 0$ one has $f(\eta_i) = [\eta_{i+s}]$ and $f(\chi_i) = \chi_{i+s}$ (modulo homotopy relative endpoints), where all indices are understood to count only modulo $q$.
\end{cor}

\begin{proof}
This is a direct consequence of Proposition \ref{from-LUbig} (2), since 
in that statement the 
cyclically 
legal 
concatenation $[f^t(\gamma)]$ must be built only from INP paths $\eta_i$ 
and from connecting zero paths. 
Indeed, 
except for subpaths of 
those $\eta_i$, 
the path $[f^t(\gamma)]$ can not contain any 
legal 
edge 
path: otherwise, since $f$ is expanding, the loop $\gamma$ in question could not be fixed by~$f$.
Since 
by hypothesis 
$[g]$ is not contained in any of the vertex spaces, there must be at least one such INP path in the legal concatenation $[f^t(\gamma)]$.
From the normal form for reduced words in a graph-of-groups it follows that $f$ acts as cyclic permutation on the cyclically reduced word in 
the Bass group $\Pi(\cal G_X)$ 
which represents 
the fixed conjugacy class 
$[g]$, and since $f$ maps any INP path (up to reduction at the tip) to  an INP path, we deduce the ``only if'' direction of the statement. The ``if'' direction is trivially true.
\end{proof}

\begin{rem}
\label{3-comments}
Some comments about Corollary \ref{fixed-conj} seem appropriate:

\smallskip
\noindent
(1)
Since there are only finitely many INP paths in $\cal G$, and 
these 
are (up to reduction at the tip) permuted by the map $f$, there is a positive power $f^t$ which fixes each of them. Through 
possibly replacing $t$ by a larger exponent 
(which we still call $t$) 
we can also 
assume, for a given cyclically legal concatenation $\gamma$ as in (\ref{with-s}), that 
each of the connecting zero paths $\chi_i$ is fixed (up to homotopy rel. 
endpoints). 
Hence for the initial and terminal vertex space $X_v$ of $\gamma$ 
we see that in the associated graph-of-groups $\cal G_X$ 
the element in 
$\FN = \pi_1(\cal G_X, v)$, which is 
represented by 
$\gamma$, 
is fixed by the automorphism 
$\Phi = f_{*, v} \in \Aut(\FN)$ 
defined by $f^t$. Furthermore, for this iterate $f^t$ instead of $f$ we can set $s$ = 0 in the statement of Corollary \ref{fixed-conj}.

\smallskip
\noindent
(2)
For general $f$, however, the ``variable'' $s$ in the statement of Corollary \ref{fixed-conj} is needed. Without it, the 
``only if'' implication in that 
statement is in general wrong.

\smallskip
\noindent
(3)
The sentence right before Corollary \ref{fixed-conj} is a slight overstatement: In fact, the algorithmic input from the factors needed there is slightly stronger, in that not just closed paths but also paths with possibly distinct given 
endpoints are concerned.
\end{rem}

\begin{rem}
\label{periodic-fixed}
From the finiteness arguments stated in Remark \ref{3-comments} (1) we see directly that any conjugacy class $[g']$ in $\pi_1 \cal G$, which is not fixed by $f_*$ but rather mapped periodically, is also represented by a concatenation $\gamma$ as in (\ref{with-s}), except that not $f$ but only some positive power of $f$ acts as cyclic permutation on this concatenation.
\end{rem}

\section{Absolute train track maps with and without periodic edges}
\label{sec:tts-with-periodics}

In the special case where $f$ is an 
expanding 
absolute train track map (i.e. all vertex groups 
$G_v = \pi_1 X_v$ are trivial), the statement of Theorem \ref{main} is well known, see for instance \cite{LU-big}, \S 4.2.
Nevertheless, even a small modification of the hypotheses will perturb this result: If one considers an absolute train track map $f: \Gamma \to \Gamma$,  for a finite graph $\Gamma$, and admits edges that are not expanded by $f$ but mapped periodically, then in general there will be infinitely many distinct INPs in $\Gamma$. Indeed, for any such periodic edge $e$ which terminates at the initial vertex of an INP path $\eta$, the concatenation $e \circ \eta$ will also be an INP path.

\smallskip

In this case, however, we can build a special graph-of-spaces $\cal G$ from $\Gamma$ by partitioning the edges of $\Gamma$ as follows: 
An edge $e$ is {\em polynomially growing} if there exist integers $c, d \geq 0$ such that one has $|f^t(e)| \leq c t^d$ for all $t \geq 1$. Otherwise $e$ is called {\em exponentially growing}. We then define the relative part $X$ of $\cal G$ to be the union of all polynomially growing edges 
and of all vertices of $\Gamma$, 
so that each vertex space $X_v$ is a connected component of 
this union, 
while every 
graph-of-spaces edge of $\cal G$ 
(as defined in 
section \ref{prelims}) 
is an exponentially growing edge of $\Gamma$. Now $f$ defines a train track map 
$f_\cal G: \cal G \to \cal G$ 
on the graph-of-spaces $\cal G$ which is indeed expanding, so that we can now apply Theorem \ref{main} to deduce the desired finiteness results modulo the polynomially growing part $X$ of $\Gamma$.

\begin{rem}
\label{beta-special}
In the 
complete 
absence of INP paths in $\cal G$, 
the 
train track map 
$f_\cal G: \cal G \to \cal G$ obtained this way 
turns out to be a special case of 
a $\beta$-train-track map as defined in 
 \S 4 of \cite{Lu-PF}. 
Such $f_\cal G$ are 
easier to deal with than the general case
of $\beta$-train-track maps, 
which carries a certain amount of technical baggage that might appear heavy at first sight. 
The ``absolute'' $\beta$-train-track maps $f_\cal G$ 
defined 
above deserve special interest, 
in that all the issues that come into play when trying to study the structure of an automorphisms 
$\phi \in \Out(\FN)$ 
from a given 
train track map (see \S7 of \cite{Lu-PF}) can already be seen here in a nutshell.
\end{rem}

\begin{rem}
\label{fix-graph}
(1)
In order to simplify the discussion, we will from now on assume that every polynomially growing edge of $\Gamma$ is actually fixed by some positive power of 
$f$. Hence 
the polynomially growing subspace $X$ of $\Gamma$
consists 
only of 
periodic 
edges and of all vertices of 
$\Gamma$, and we 
also assume that $\Gamma$ has been subdivided 
such 
that every INP path $\eta$ terminates in vertices. We then build a graph $X^*$ from $X$ by 
adding 
for each of the finitely many INP paths $\eta$ of $\cal G$ an edge $e_\eta$ to 
$X$, 
which has the same endpoints as $\eta$. On the resulting graph $X^*$ we define a map $f^*$ by declaring $f^*(P) = f(P)$ for any vertex $P$ of $X^*$, \, $f^*(e_\eta) := e_{[f(\eta)]}$ for any INP path $\eta$, and $f^*(e) = f(e)$ for all other edges of $X^*$.

\smallskip
\noindent
(2) 
We define a graph map $h: X^* \to \Gamma$ which acts as identity map on every vertex or edge that belongs to both, $X^*$ and $\Gamma$, and satisfies 
$h(e_\eta) = \eta$ for any INP path $\eta$. 
We observe that one has $h f^*(\gamma) = [f h(\gamma)]$ for every edge path $\gamma$ in $X^*$.

\smallskip
\noindent
(3) 
It follows now directly from Corollary \ref{fixed-conj} that a conjugacy class in $\pi_1 \Gamma$ is periodic if and only if it is represented by a loop in image of the map $h$. In particular, one deduces the following 
proposition.
\end{rem}

\begin{prop}
\label{direct-cor}
Let $f: \Gamma \to \Gamma$, \, $f^*: X^* \to X^*$ and $h: X^* \to  \Gamma$ be as in Remark \ref{fix-graph}. 
Assume furthermore that every periodic edge of $\Gamma$ is actually fixed by $f$, 
and that for every INP path $\eta$ one has $[f(\eta)] = \eta$. 

A representative $\Phi \in \Aut(\FN)$ of the outer automorphisms $\phi := f_*$ induced by $f$ on $\FN := \pi_1 \Gamma$ has non-trivial fixed subgroup $Fix(\Phi)$ if and only if one of the following takes place:
\begin{enumerate}
\item
There is an $f^*$-fixed vertex $P \in X^{*}$ which lies in some non-contractible component $X^{*}_P$ of $X^{*}$, such that, if $\pi_1 \Gamma$ is read off from the base point $h(P)$, then 
the 
lift $\Phi$ of $\phi$ is 
defined 
through: 
$$\Phi = f_{*,h(P)}: \pi_1(\Gamma, h(P)) \to \pi_1(\Gamma, h(P))$$
In this case one has
$$
Fix(\Phi) = h_{*, P}(\pi_1 (X^{*}_P, P))\, ,$$
and a finite system of generators for $Fix(\Phi)$ can be determined algorithmically from the graph $\Gamma$ and the map $f$.
\item
The representative $\Phi$ of $\phi$ is obtained from another lift $\Phi'$ which is as in (1) above, through composition of $\Phi'$ with a conjugation by some non-trivial element $w \in Fix(\Phi')$. In this case $Fix(\Phi)$ is the cyclic group generated by the element $w$.
\end{enumerate}
\end{prop}

\begin{proof}
We first recall that the action of the fundamental group $\FN = \pi_1 \Gamma$ 
on the universal covering $\tilde \Gamma$ determines for any map $f: \Gamma \to \Gamma $ a bijection between the lifts $\tilde f: \tilde \Gamma \to \tilde \Gamma $ on one hand, and the representatives $\Phi \in \Aut(\FN)$ of the outer automorphism $\phi = f_*$ 
on the other. 
This bijection is given by the equality
\begin{equation}
\label{lifts}
\Phi(w) \tilde f = \tilde f w: \tilde \Gamma \to \tilde \Gamma \qquad \text{for any $w \in \FN$}\, .
\end{equation}
Since $\tilde \Gamma$ is a simplicial tree, every non-trivial $w \in \FN$ fixes set-wise an axis $ax(w)$ in $\tilde \Gamma$, and 
$w$ 
shifts along this axis as a translation. From (\ref{lifts}) 
and the uniqueness of this axis 
we deduce that $w$ belongs to $Fix(\Phi)$ if 
and only if 
$\tilde f$ 
fixes set-wise
(modulo reduction) the axis of $w$,
while 
preserving its orientation. 
This axis 
$ax(w)$ is mapped under the covering map $\tilde \Gamma \to \Gamma$ to the loop $\gamma$ that represents the conjugacy class $[w]$.
We 
see from Corollary \ref{fixed-conj} that either 
$\gamma$ 
consists only of edges from $X$, or else it must run over at least one INP path. In either case, there is a vertex $P$ 
which is fixed by $f$, over which the loop $\gamma$ must cross. 
If any lift $\tilde P$ of $P$ on $ax(w)$ is fixed by $\tilde f$, then case (1) of the claim occurs; otherwise we are in case (2). This is a direct consequence of the 
classical 
configuration described here, and of our definition of $X^*$ and of $f^*$ in Remark \ref{fix-graph}.
\end{proof}

\begin{rem}
\label{periodic-non-fixed}
Let us consider the statement of 
Proposition \ref{direct-cor}, but without the extra assumption that every periodic edge 
and every INP path is fixed by $f$ (up to reduction at the tip of the INP paths).
\begin{enumerate}
\item
In this case 
one can replace $f$ by a positive power $f^t$ which {\em does} satisfy this 
extra 
assumption. The existence of such an exponent $t \geq 1$ follows 
directly from 
the finiteness of $\Gamma$, and 
from 
the finiteness 
(shown in Proposition \ref{INP-finite}) 
of the set of INP paths in the graph-of-spaces $\cal G$ associated as above to $\Gamma$.
Hence the conclusion in Proposition \ref{direct-cor} holds, for $f^t$ instead of~$f$.
\item
On the other hand, the implication  
of Proposition \ref{direct-cor}, now considered again for $f$, but with $X^*$ replaced by the subgraph $X^{**}$ of all edges of $X^*$ that are fixed by $f^*$, 
turns out to be 
wrong: in this case there are in general fixed conjugacy classes as in Corollary \ref{fixed-conj} which are represented by a path $\gamma$ as in (\ref{with-s}) with 
$1 \leq s \leq q-1$. 
Such $\gamma$ defines a non-trivial element of $\FN$ which is fixed by a representative $\Phi$ of $\phi$ 
that 
does not belong to either of the classes (1) of (2) from Proposition \ref{direct-cor}, with $X^{**}$ replacing $X^*$.
\end{enumerate}
\end{rem}

Going back to Remark \ref{beta-special}, 
the determination of the 
``structure'' 
of a given outer automorphism $\phi$ 
is in full detail a very challenging task. To give the reader a taste of the subtleties that come into play we now assume that $\phi$ is represented by an absolute train track map $\phi$. In addition we also assume that the transition matrix $M(f)$ (see 
\cite{CL} \S2) is primitive: in other words, there exists an exponent $t \geq 1$ such that for any edge $e_i$ of $\Gamma$ the edge path $f^t(e_i)$ crosses over every edge $e_j$ of~$\Gamma$. 

Note 
that 
the primitivity assumption for $M(f)$ 
implies 
for any vertex space $X_v$ 
in the above defined associated expanding train track map 
$f: \cal G \to \cal G$, 
since every edge of $X_v$ is by definition polynomially growing, that $X_v$ is  
trivial (meaning it consists of a single vertex only).
The converse implication doesn't hold, 
since 
the absence of polynomially growing edges only implies that $f: \Gamma \to \Gamma$ is expanding, but 
in general there is more than one stratum in 
an expanding absolute train track map.

\begin{rem}
\label{near-iwip}
Let $\cal C_{prim} = \cal C_{prim}(\FN)$ denote 
the class of 
all automorphisms $\phi \in \Out(\FN)$
which can be 
represented by an expanding absolute train track map $f: \Gamma \to \Gamma$ with primitive transition matrix.

\smallskip
\noindent
(1)
It is well known that $\cal C_{prim}$ contains all iwip autormorphims (sometimes also called ``fully irreducible''), including the toroidal ones, i.e. those which are induced by pseudo-Anosov homeomorphisms of a surface with one boundary component.

\smallskip
\noindent
(2)
It is also fairly well known that there are 
in addition automorphisms $\phi \in \cal C_{prim}$ which are not iwip, for instance if $\phi$ derives from a pseudo-Anosov homeomorphism of a surface with more than one boundary component. 

\smallskip
\noindent
(3)
One typically attributes the existence of non-iwip $\phi \in \cal C_{prim}$ to the failure of a local iwip-criterion, namely the connectedness of the Whitehead graph on each vertex of $\Gamma$, which is canonically defined by $f$ (see 
\cite{CL}). 
However, the examples in (2) above have in general connected vertex Whitehead graphs.

\smallskip
\noindent
(4)
From (2) and (3) above we deduce that, other than the connectedness of the vertex Whitehead graphs, the second main obstruction of any $\phi \in \cal C_{prim}$ to be iwip is the presence of 
more than one INP loop 
as in 
(\ref{with-s}). 
Indeed, armed with the information from Proposition \ref{direct-cor}, it 
is not hard 
to construct expanding absolute train track map $f: \Gamma\to \Gamma$ 
which give rise to automorphisms $\phi = f_* \in \cal C_{prim}$ that fix (up to conjugation) subgroups of $\FN = \pi_1 \Gamma$ which are of rank $\geq 2$.
However, the author admits freely that, adding the requirement that all vertex Whitehead graphs are connected, left him seriously in doubt whether such $f$ can indeed exist. But the following example settles this question.
\end{rem}

\begin{example}
\label{exotic-near-iwip}
We consider the following automorphism of  $F_6 = F(a,b,c,d,e,f,)$ and its realization $f: \cal R_6 \to \cal R_6$ on the rose $\cal R_6 = \cal R(a, b, c, d, e, f)$:
\begin{equation}
\label{exotic}
	a  \,\,\, \mapsto \,\,\,  b\, a\, c\, c\, d\, d\, e\, e\, f\, f \,, \quad
	b  \,\,\, \mapsto \,\,\,  a\,, \quad
	c  \,\,\, \mapsto \,\,\,  a \,c\,, \quad
	d  \,\,\, \mapsto \,\,\,  d\, a\,, \quad
	e  \,\,\, \mapsto \,\,\,  a \,e\,, \quad
	f  \,\,\, \mapsto \,\,\,  f\, a
\end{equation}
One sees directly that the transition matrix $M(f)$ is primitive, and it doesn't take long to verify that the Whitehead graph at the sole vertex of $\cal R_6$ is connected. Since the automorphism is positive, the map $f$ is absolute train track, and since every generator has length $\geq 2$, it is expanding. 
However, 
all elements of the subgroup generated by  $c^{-1}\, e$  and  $d \,f^{-1}$  are fixed.

The ``complicated'' image of the generator $a$ 
in (\ref{exotic}) defines 
only a special case of a whole family of automorphisms which all have the 
above 
properties: One can alternatively set
$$	a  \,\,\, \mapsto \,\,\,  b \,W	$$
for any word  $W$  which avoids the generator $b$,  and is complicated enough to guarantee the connectedness of the Whitehead graph. Avoiding  $b$  makes  sure that we have an automorphism at hand, and not just an endomorphism.  
We thus observe that there are indeed plenty of such ``exotic'' examples.
\end{example}

From Remark \ref{near-iwip} we see that the class $\cal C_{prim}$ 
is still too large to expect that each of its elements 
strongly resembles an iwip automorphism. 
On the other hand, certain fundamental properties of iwip automorphisms (like the existence of a unique expanding invariant $\R$-tree) are shared by some of the automorphisms $\phi \notin \cal C_{prim}$ which are captured by Proposition \ref{direct-cor}.
In joint work 
\cite{KL} 
with I. Kapovich 
(also \cite{Lu-tree-like}) 
the author has made an attempt to 
systematically 
study automorphisms that are ``near-iwips'', 
with respect to the following particular 
aspects:
\begin{enumerate}
\item
Uniqueness and properties of the projectively fixed expanding limit $\R$-tree.
\item
Fixed free factors or fixed free conjugacy classes.
\item
The existence of very special train track representatives.
\item
The uniqueness of the attracting lamination, and its particular properties.
\end{enumerate}
In fact, a first version of Proposition \ref{direct-cor} already transpired in  discussions Kapovich and I had at various times around \cite{KL}. It turns out, however, that it is difficult to conciliate 
the above issues (1) - (4) 
into 
a single 
convincing concept.

\end{document}